\documentclass[a4paper]{amsart}

\usepackage{a4wide,amsfonts,amssymb,amsmath,color,todonotes,tikz-cd,,mathtools}
\usepackage[shortlabels]{enumitem}
\usepackage[english]{babel}

\allowdisplaybreaks

\DeclareFontFamily{U}{mathx}{}
\DeclareFontShape{U}{mathx}{m}{n}{<-> mathx10}{}
\DeclareSymbolFont{mathx}{U}{mathx}{m}{n}
\DeclareMathAccent{\widehat}{0}{mathx}{"70}
\DeclareMathAccent{\widecheck}{0}{mathx}{"71}

\newtheorem{lem}{Lemma}[section]
\newtheorem{defi}[lem]{Definition}
\newtheorem{theo}[lem]{Theorem}

\newtheorem{rem}[lem]{Remark}

\def\equi{\Leftrightarrow}
\def\qequi{\quad\equi\quad}

\newcommand{\wh}[1]{\widehat{#1}}
\newcommand{\wt}[1]{\widetilde{#1}}
\newcommand{\wc}[1]{\widecheck{#1}}

\def\n{\mathbb{N}}

\def\reals{\mathbb{R}}

\def\rt{\reals^{3}}

\def\rN{\reals^{N}}
\def\om{\Omega}
\def\ga{\Gamma}
\def\gat{\ga_{\mathsf{t}}}
\def\gan{\ga_{\mathsf{n}}}

\def\cd{c_{\Delta}}
\def\cdd{c_{\Delta^{2}}}
\def\cf{c_{\mathsf{f}}}
\def\cp{c_{\mathsf{p}}}
\def\cfp{c_{\mathsf{fp}}}
\def\ca{c_{\A}}
\def\bbS{\mathbb{S}}
\def\bbT{\mathbb{T}}
\def\sfL{\mathsf{L}}
\def\sfP{\mathsf{P}}
\def\sfH{\mathsf{H}}
\def\sfC{\mathsf{C}}
\def\calR{\mathcal{R}}
\def\calH{\mathcal{H}}
\def\bbH{\mathbb{H}}
\def\bbBH{\mathbb{BH}}

\def\ttD{\mathtt{D}}
\def\ttN{\mathtt{N}}

\newcommand{\Harm}{\calH}
\newcommand{\Harmd}{\Harm_{\ttD}}
\newcommand{\Harmn}{\Harm_{\ttN}}

\newcommand{\Leb}{\sfL}
\newcommand{\Lt}{\Leb^{2}}
\newcommand{\hLt}{\wh{\Leb}^{2}}
\newcommand{\tLt}{\wt{\Leb}^{2}}
\newcommand{\cLt}{\wc{\Leb}^{2}}
\newcommand{\hhLt}{\wh{\wh{\Leb}}^{2}}

\renewcommand{\P}{\sfP}
\renewcommand{\H}{\sfH}
\newcommand{\Hc}{\mathring{\sfH}}

\newcommand{\hH}{\wh{\H}}
\newcommand{\tH}{\wt{\H}}
\newcommand{\cH}{\wc{\H}}
\newcommand{\hhH}{\wh{\wh{\H}}}
\newcommand{\C}{\sfC}
\newcommand{\Cc}{\mathring{\sfC}}

\DeclareMathOperator{\tr}{tr}
\DeclareMathOperator{\p}{\partial}
\DeclareMathOperator{\id}{id}
\DeclareMathOperator{\na}{\nabla}
\DeclareMathOperator{\nac}{\mathring{\na}}
\DeclareMathOperator{\naga}{\na_{\ga}}
\DeclareMathOperator{\naes}{\na_{\emptyset}}
\DeclareMathOperator{\rot}{rot}
\DeclareMathOperator{\rotc}{\mathring{\rot}}
\def\div{\operatorname{div}}
\DeclareMathOperator{\divc}{\mathring{\div}}
\DeclareMathOperator{\divga}{\div_{\ga}}
\DeclareMathOperator{\dives}{\div_{\emptyset}}

\DeclareMathOperator{\Div}{Div}
\DeclareMathOperator{\divDiv}{divDiv}
\DeclareMathOperator{\Rot}{Rot}
\DeclareMathOperator{\symRot}{symRot}
\DeclareMathOperator{\devna}{dev\na}
\DeclareMathOperator{\nana}{\na\!\na}
\DeclareMathOperator{\nanaga}{\nana_{\ga}}
\DeclareMathOperator{\nanaes}{\nana_{\emptyset}}

\def\A{\operatorname{A}}

\def\cA{\operatorname{\mathcal{A}}}

\def\B{\operatorname{B}}
\def\Bc{\mathring{\B}}
\def\cB{\operatorname{\mathcal{B}}}
\def\cBc{\mathring{\cB}}
\def\Bd{\B_{\ttD}}
\def\Bn{\B_{\ttN}}
\def\cBd{\cB_{\ttD}}
\def\cBn{\cB_{\ttN}}
\def\L{\operatorname{L}}
\def\Lc{\mathring{\L}}
\def\cL{\operatorname{\mathcal{L}}}
\def\cLc{\mathring{\cL}}
\def\Ld{\L_{\ttD}}
\def\Ln{\L_{\ttN}}
\def\cLd{\cL_{\ttD}}
\def\cLn{\cL_{\ttN}}
\def\Bdd{\B_{\ttD,\ttD}}
\def\Bnn{\B_{\ttN,\ttN}}
\def\Bdn{\B_{\ttD,\ttN}}
\def\Bnd{\B_{\ttN,\ttD}}
\def\cBdd{\cB_{\ttD,\ttD}}
\def\cBnn{\cB_{\ttN,\ttN}}

\def\cBnd{\cB_{\ttN,\ttD}}
\newcommand{\norm}[1]{|#1|}
\newcommand{\bnorm}[1]{\big|#1\big|}
\newcommand{\scp}[2]{\langle#1,#2\rangle}

\title[Biharmonic Equations]
{Biharmonic Equations}
\author{Dirk Pauly}
\address{Institut f\"ur Analysis, Technische Universit\"at Dresden, Germany}
\email[Dirk Pauly]{dirk.pauly@tu-dresden.de}
\author{Alberto Valli}
\address{Department of Mathematics, University of Trento, Italy}
\email[Alberto Valli]{alberto.valli@unitn.it}
\subjclass{???}
\date{\today}
\thanks{{\it Corresponding Author}. Dirk Pauly}

\begin{document}


\begin{abstract}
In this note we devise and analyse well-posed variational formulations
and operator theoretical methods for boundary value problems 
associated to the biharmonic operator $\Delta^{2}$. 
Of particular interest are Neumann type and over-/underdetermined
(maximal/minimal) boundary value problems.
\end{abstract}


\vspace*{-10mm}
\maketitle
\setcounter{tocdepth}{3}
\tableofcontents



\section{Introduction}
\label{sec:intro}

Throughout this paper -- unless otherwise stated -- let $\om\subset\rN$, $N\in\n$, 
be a domain (connected and open set) with boundary $\ga:=\p\om$
for which we assume no regularity at all.

Recently, in the new edition of the book of Valli \cite{AlbV}, the second author discussed 
a Neumann boundary value problem for the biharmonic equation
and presented proofs for the well-posedness of related variational formulations
under very strong regularity assumptions ($\C^{4}$-boundary),
cf.~\cite[Section 5.6]{AlbV}.

In this context some interesting questions arose
which we will answer in this contribution using basic methods
from functional analysis, cf.~Section \ref{sec:tinyfatoolbox}.
In particular, we can reduce the regularity to Lipschitz boundaries and even beyond.
More precisely, we shall show that even under very weak regularity and boundedness assumptions on the underlying domain
(admissible domains, cf.~Definition \ref{def:admdom1} and Remark \ref{rem:admdom1})
we can answer the latter questions and extend the results to a whole zoo
of biharmonic operators, including, among others, the Neumann case.

The two relevant underlying Hilbert complexes are the de Rham complex
\begin{equation*}
\def\arrowlength{10ex}
\def\arrowdistance{.8}
\begin{tikzcd}[column sep=\arrowlength]
\Lt(\om) 
\ar[r, rightarrow, shift left=\arrowdistance, "\nac"] 
\ar[r, leftarrow, shift right=\arrowdistance, "-\div"']
&
\Lt(\om) 
\ar[r, rightarrow, shift left=\arrowdistance, "\rotc"] 
\ar[r, leftarrow, shift right=\arrowdistance, "\rot"']
& 
\Lt(\om) 
\arrow[r, rightarrow, shift left=\arrowdistance, "\divc"] 
\arrow[r, leftarrow, shift right=\arrowdistance, "-\na"']
& 
\Lt(\om),
\end{tikzcd}
\end{equation*}
cf.~\eqref{complex:derham1}, and the Hessian complex
\begin{equation*}
\def\arrowlength{10ex}
\def\arrowdistance{.8}
\begin{tikzcd}[column sep=\arrowlength]
\Lt(\om) 
\ar[r, rightarrow, shift left=\arrowdistance, "\mathring{\nana}"] 
\ar[r, leftarrow, shift right=\arrowdistance, "\divDiv_{\bbS}"']
&
\Lt_{\bbS}(\om) 
\ar[r, rightarrow, shift left=\arrowdistance, "\mathring{\Rot}_{\bbS}"] 
\ar[r, leftarrow, shift right=\arrowdistance, "\symRot_{\bbT}"']
& 
\Lt_{\bbT}(\om) 
\arrow[r, rightarrow, shift left=\arrowdistance, "\mathring{\Div}_{\bbT}"] 
\arrow[r, leftarrow, shift right=\arrowdistance, "-\devna"']
& 
\Lt(\om),
\end{tikzcd}
\end{equation*}
cf.~\eqref{complex:biharm}.

For simplicity and readability we restrict our analysis to homogeneous boundary conditions.
Note that as soon as proper trace and extension operators are available,
the inhomogeneous boundary value problems can be easily formulated as homogeneous ones.
A very general trace theory, meeting our needs, 
has been developed recently
by the first author and his collaborators in \cite{HPS2023a}.

For better readability of the introductory part
and to get quickly to the heart of matter,
most exact definitions of spaces and operators will be postponed 
to Section \ref{sec:optheo}.
Nevertheless we shall use our notations already in this introduction.
Throughout this paper we shall use, among the classical Sobolev spaces such as
$$\H^{1}(\om),
\qquad
\H(\rot,\om),
\qquad\H(\div,\om),$$
the Sobolev spaces
\begin{align*}
\H(\Delta,\om)&:=\big\{\varphi\in\Lt(\om):\Delta\varphi\in\Lt(\om)\big\},
&
\H(\Delta^{2},\om)&:=\big\{\varphi\in\Lt(\om):\Delta^{2}\varphi\in\Lt(\om)\big\},
\intertext{with kernels (harmonic resp. biharmonic functions)}
N(\Delta)=\bbH&:=\big\{\varphi\in\Lt(\om):\Delta\varphi=0\big\},
&
N(\Delta^{2})=\bbBH&:=\big\{\varphi\in\Lt(\om):\Delta^{2}\varphi=0\big\},
\end{align*}
and restrictions (projections orthogonal to the kernels)
\begin{align*}
\hLt(\om)&:=\Lt(\om)\cap\reals^{\bot},
&
\qquad\tLt(\Delta,\om)&:=\Lt(\om)\cap\bbH^{\bot},
&
\cLt(\Delta^{2},\om)&:=\Lt(\om)\cap\bbBH^{\bot},\\
\hH^{1}(\om)&:=\H^{1}(\om)\cap\reals^{\bot},
&
\qquad\tH(\Delta,\om)&:=\H(\Delta,\om)\cap\bbH^{\bot},
&
\cH(\Delta^{2},\om)&:=\H(\Delta^{2},\om)\cap\bbBH^{\bot},
\end{align*}
where $\bot$ denotes orthogonality in $\Lt(\om)$.

\subsection{Historical Remarks}
\label{sec:introhistory}

Let us consider a Neumann boundary value problems for the biharmonic operator,
that is
$$\Delta^{2}u=f\text{ in }\om
\qquad\text{and}\qquad
\Delta u=0,\quad
n\cdot\na\Delta u=0\text{ on }\ga,$$
more closely.
Later, the corresponding proper Neumann biharmonic operator will be denoted by
$$\cBn
=\cB_{\circ,\cdot}
=\cLc\cL,$$
see Section \ref{sec:zoo1}.
From the physical point of view this problem
is not the most interesting, as the model which describes the equilibrium position of an elastic thin plate, 
unconstrained on the boundary, involves other second order and third order boundary operators, 
in which the Poisson ratio also has a role (see Courant and Hilbert \cite[p.~250]{couhil} 
and more recently, e.g., Verchota \cite{ver}, Provenzano \cite{prov}; 
the original physical model even dates back to Kirchhoff and Kelvin). 

Instead of 
$$\Delta^{2}=\Delta\Delta,$$
the bi-Laplacian can also be factorised by means of 
$$\Delta^{2}=\divDiv\nana,$$
cf.~\cite{PZ2020a,PS2024a},
with different Neumann boundary conditions and hence different physical interpretations.
We shall come back later to this alternative,
cf.~Section \ref{sec:introalt} and Section \ref{sec:appbiharm}.

However, despite these remarks, we think that the Neumann problem for the biharmonic operator 
has a nice and simple mathematical structure, similar to that of other classical problems, 
and we find it interesting from the mathematical point of view. 
Moreover, it is the limiting case, for the Poisson ratio going to $1$ of suitable physical models.

Let us first compare a bit the Neumann problems for $-\Delta$ and $\Delta^{2}$.
As it is well-known, the Neumann boundary value problems 
for the negative Laplacian $-\Delta$
and what we call the bi-Laplacian/biharmonic operator $\Delta^{2}$,
i.e.,
\begin{align*}
-\Delta v&=g,
&
\Delta^{2}u&=f
&
&\text{in }\om,\\
n\cdot\na v&=0,
&
\Delta u&=0
&
&\text{on }\ga,\\
&&
n\cdot\na\Delta u&=0,
&
&\text{on }\ga,
\end{align*}
have similarities and differences.
In particular, for both of them the solution is not unique: 
adding to $v$ a constant $c\in\reals$ resp.~to $u$ a harmonic function $h\in\bbH$
gives another solution. 
Moreover, for both of them the data have to satisfy a compatibility condition
(Fredholm alternative):  
we have 
$$g\bot\,\reals\quad\text{for }-\Delta
\qquad\text{and}\qquad
f\bot\,\bbH\quad\text{for }\quad\;\Delta^{2}.$$
Note that in the smooth case we have
\begin{align*}
\scp{g}{1}_{\Lt(\om)}
&=-\int_{\om}\div\na v
=-\int_{\ga}n\cdot\na v
=0,\\
\scp{f}{h}_{\Lt(\om)}
&=\int_{\om}(\Delta^{2}u)h
=\int_{\ga}(n\cdot\na\Delta u)h
-\int_{\ga}\Delta u\,(n\cdot\na h)
=0.
\end{align*}

A first difference concerns the type of the boundary conditions: 
the Neumann boundary condition for $-\Delta$ satisfies the so-called complementing condition 
of Agmon, Douglis, and Nirenberg \cite{ADN}, 
that in the present situation simply says that the polynomial 
$t$ is not divisible by $t-i$.
This is not true for the Neumann boundary condition for $\Delta^{2}$: 
the complementing condition would require that the polynomials $1+t^{2}$ and $t+t^{3}$ 
were linearly independent modulo $(t-i)^{2}$, and it is well-known that this is not the case.
For the ease of the reader, let us readily show this last statement: we have
$$\frac{1+t^{2}}{(t-i)^{2}}
=1+\frac{2(1+it)}{(t-i)^{2}},\qquad
\frac{t+t^{3}}{(t-i)^{2}}=t+2i+\frac{2(i-t)}{(t-i)^{2}},$$
and the second remainder $2(i-t)$ is proportional to the first remainder $2(1+it)$ by a factor $i$.
Note that the difference between the two Neumann problems 
still shows up when considering the so-called Lopatinski\u{\i}--\v{S}apiro condition 
(see Wloka \cite[Examples 11.2 and 11.8]{Wloka}); 
in fact, it is known that the complementing condition and the Lopatinski\u{\i}--\v{S}apiro 
condition are equivalent (see, e.g., Negr\'on-Marrero and Montes-Pizarro \cite[Appendix A]{NegMon}).

Another difference seems to be related to well-posedness: in fact, 
the Neumann boundary value problem associated to $-\Delta$ is well-posed in a suitable space 
where uniqueness is recovered and for data which satisfy the necessary compatibility condition
(Fredholm alternative). 
This well-known result is an easy consequence of the Poincar\'e inequality,
cf.~Definition \ref{def:admdom1},
and the Riesz representation theorem, or of more elaborate arguments 
using linear functional analysis, cf.~the discussion of the Neumann Laplacian $\cLn$ from Section \ref{sec:standlap}.
On the contrary, well-posedness for the Neumann boundary value problem associated to the $\Delta^{2}$
operator seemed to be questionable so far 
(see, e.g., what is explicitly reported in Verchota \cite[p.~217 and Sc.~21]{ver}, 
and in a more indirect way in Renardy and Rogers \cite[Sc.~9.4.2 and Example 9.30]{RenRog}, 
Gazzola, Grunau and Sweers \cite[Sc.~2.3]{GGS}, 
Provenzano \cite[p.~1006]{prov}). 
In addition to this, it can be noted that Begehr \cite{beg} presents a long list of \emph{twelve} boundary value problems 
for the biharmonic operator that are either well-posed or solvable under suitable compatibility conditions, 
and in that list the Neumann problem is not included. 

Going a little bit more in depth, in Renardy and Rogers \cite{RenRog}, 
Gazzola, Grunau and Sweers \cite{GGS}, and Provenzano \cite{prov} 
the comments about the fact that the Neumann problem for the biharmonic operator $\Delta^{2}$ 
is possibly not well-posed are related to the fact that the complementing condition is not satisfied 
(in particular, this condition is assumed in the existence and uniqueness 
Theorems 2.16 and 2.20 in \cite{GGS}; there see also Remark 2.17). 
This is apparently to be meaningful, as in Agmon, Douglis and Nirenberg \cite[Sc.~10]{ADN} 
it is explicitly proved that the complementing condition is necessary 
for obtaining higher order a-priori estimates in H\"older and $\Leb^{p}$ spaces 
(in this respect, see also Lions and Magenes \cite[Chap.~2, Sc.~8.3 and Remark 9.8]{LioMag}). 

However, rather surprisingly, 
it turns out that this condition is not necessary for well-posedness 
in suitable Hilbert spaces, as the following example shows:
Consider the operator
$$\Delta^{2}+1.$$
Since the complementing condition only depends on the principal parts of the spatial and boundary operators, 
we are in the same situation of the Neumann problem for the biharmonic operator; 
therefore the complementing condition is also not satisfied.
However, the weak formulation of the problem (with homogeneous boundary data) reads:
For $f\in\Lt(\om)$ find $u\in\H(\Delta,\om)$ such that
\begin{align}
\label{eq:V}
\forall\,\varphi\in\H(\Delta,\om)
\qquad
\scp{\Delta u}{\Delta\varphi}_{\Lt(\om)}
+\scp{u}{\varphi}_{\Lt(\om)}
=\scp{ f}{\varphi}_{\Lt(\om)},
\end{align}
which is uniquely solvable by Riesz' representation theorem
with $\norm{u}_{\H(\Delta,\om)}\leq\norm{f}_{\Lt(\om)}$.
Being now evident that the complementing condition is not necessary for well-posedness in a suitable Hilbert space, 
in this paper we shall show that indeed, among others, the Neumann boundary value problem 
for the biharmonic operator is well-posed.

\subsection{Approach by Bi-Laplacians}
\label{sec:introbilap}

The approach in \cite[Section 5.6]{AlbV} is based on the representation 
of the biharmonic operator as bi-Laplacian $\Delta^{2}$ with 
Dirichlet/Neumann type and over-/underdetermined boundary conditions.

The story begins with the (full, not mixed) Dirichlet and Neumann boundary conditions for the negative Laplacian $-\Delta$.
These are the two reasonable boundary conditions for 
$$-\Delta=-\div\na,$$
resulting in the apparently four ``reasonable'' (full) boundary conditions 
for the biharmonic operator interpreted as bi-Laplacian.
More precisely:
For a possible solution $u$ of 
$$\Delta^{2}u=f\qquad\text{in }\om$$ 
we may impose a couple of the boundary conditions
\begin{align}
\label{eq:bcD}
u&=0
&
&\text{on }\ga,
&
&\text{\big(Dirichlet on $u$, i.e., $u\in\Hc^{1}(\om)$\big)}\\
\label{eq:bcN}
n\cdot\na u&=0
&
&\text{on }\ga,
&
&\text{\big(Neumann on $u$, i.e., $\na u\in\Hc(\div,\om)$\big)}\\
\label{eq:bcDD}
\Delta u&=0
&
&\text{on }\ga,
&
&\text{\big(Dirichlet on $\Delta u$, i.e., $\Delta u\in\Hc^{1}(\om)$\big)}\\
\label{eq:bcND}
n\cdot\na\Delta u&=0
&
&\text{on }\ga,
&
&\text{\big(Neumann on $\Delta u$, i.e., $\na\Delta u\in\Hc(\div,\om)$\big)}
\end{align}
as presented in the equations \cite[(5.24)-(5.27)]{AlbV}.
Some of these six pairs have prominent names:
\begin{align}
\begin{aligned}
\label{eq:bcdef1}
\text{Dirichlet:}&
\quad\eqref{eq:bcD}\wedge\eqref{eq:bcN}
&
\text{Navier:}&
\quad\eqref{eq:bcD}\wedge\eqref{eq:bcDD}\\
\text{Neumann:}&
\quad\eqref{eq:bcDD}\wedge\eqref{eq:bcND}
&
\text{Riquier:}&
\quad\eqref{eq:bcN}\wedge\eqref{eq:bcND}
\end{aligned}
\end{align}
As we will see later, cf.~Section \ref{sec:zoo1},
the latter four pairs
as well as $\eqref{eq:bcD}\wedge\eqref{eq:bcND}$,
then denoted by 
$$\cL\cLc,
\quad\cLc\cL,
\quad\cLd\cLd,
\quad\cLn\cLn,
\quad\cLn\cLd,$$
make $\Delta^{2}$ well-posed,
while $\eqref{eq:bcN}\wedge\eqref{eq:bcDD}$,
then denoted by $\cLd\cLn$, does not,
at least without further restrictions on the domain of definition.
The second author proposes variational formulations 
for the Dirichlet, Neumann, Navier, and Riquier
biharmonic problems,
but notices that for $\eqref{eq:bcD}\wedge\eqref{eq:bcND}$
and $\eqref{eq:bcN}\wedge\eqref{eq:bcDD}$ no well-posed
variational formulations are at hand.

\cite[Section 5.6.1]{AlbV} is then specifically devoted to the Neumann problem 
for the biharmonic equation $\eqref{eq:bcDD}\wedge\eqref{eq:bcND}$ ($\cLc\cL$).
The natural Sobolev spaces for the variational method are then given by 
$\H(\Delta,\om)$ from \eqref{eq:V}
as well as $\tH(\Delta,\om)$ 
and $\tLt(\om)$.
For $f\in\tLt(\om)$
the variational formulation is then 
to find $u\in\tH(\Delta,\om)$ such that
$$\forall\,\varphi\in\H(\Delta,\om)\;\big(\text{or }\tH(\Delta,\om)\big)\qquad
\scp{\Delta u}{\Delta\varphi}_{\Lt(\om)}
=\scp{ f}{\varphi}_{\Lt(\om)},$$
cf.~\eqref{eq:varBN} in Section \ref{sec:biharm1}.
It has been emphasised that the problem is well-posed if the
Poincar\'e type estimate
\begin{align}
\label{eq:introestDelta}
\exists\,c>0\quad\forall\,\varphi\in\tH(\Delta,\om)\qquad
\norm{\varphi}_{\Lt(\om)}\leq c\norm{\Delta\varphi}_{\Lt(\om)}
\end{align}
holds, which follows if the embedding
\begin{align}
\label{eq:introcpt}
\tH(\Delta,\om)\hookrightarrow\Lt(\om)
\end{align}
is compact. The second author could show by regularity and Rellich's selection theorem
that this is indeed true if $\om$ is bounded and of class $\C^{4}$,
since then the embedding $\tH(\Delta,\om)\hookrightarrow\H^{2}(\om)$ is continuous,
more precisely, 
\begin{align}
\label{eq:introX}
\tH(\Delta,\om)
=\Delta\big(\H^{4}(\om)\cap\Hc^{2}(\om)\big)
=\H^{2}(\om)\cap\bbH^{\bot}
\end{align}
with equivalent norms.
It has also been noted that $\H(\Delta,\om)\hookrightarrow\Lt(\om)$ is in general not compact
even for the unit ball and $N=2$.
The space $\Delta\big(\H^{4}(\om)\cap\Hc^{2}(\om)\big)$ (of high regularity)
has already been proposed by Provenzano \cite{prov} 
for studying an eigenvalue problem for the Neumann biharmonic operator. 

A theory relying on $\C^{4}$-boundaries is unsatisfying.
In this contribution we shall present a solution theory 
for quite a zoo of biharmonic problems -- including the Neumann type --
which works fine with bounded Lipschitz domains $\om$
or even beyond, cf.~the notion of admissible domains in Definition \ref{def:admdom1} and Remark \ref{rem:admdom1},
which need neither to be bounded nor smooth (Lipschitz).

\subsubsection{Sketch of Basic Ideas and Concepts}
\label{sec:introsketch}

We promote a more operator theoretical method
using basic functional analysis.
For simplicity, in this introduction, let us assume that $\om$ is a bounded Lipschitz domain.

We consider the two gradients
(minimal and maximal)
\begin{align*}
\naga:
\Hc^{1}(\om)\subset\Lt(\om)&\to\Lt(\om);
&
\phi&\mapsto\na\phi,\\
\naes:\H^{1}(\om)\subset\Lt(\om)&\to\Lt(\om),
\end{align*}
cf.~Section \ref{sec:standlap} and \cite{PS2022a}
for more details and results on the de Rham complex.
Note that we have the Friedrichs' and Poincar\'e's estimates,
cf.~Remark \ref{rem:fpest} and \eqref{eq:est1lap},
\begin{align}
\begin{aligned}
\label{eq:introfpest1}
\exists\,\cf>0
\qquad
\forall\,\varphi&\in\Hc^{1}(\om)
&
\norm{\varphi}_{\Lt(\om)}
&\leq\cf\norm{\na\varphi}_{\Lt(\om)},\\
\forall\,\varphi&\in\Hc^{2}(\om)
&
\norm{\varphi}_{\Lt(\om)}
&\leq\cf^{2}\norm{\Delta\varphi}_{\Lt(\om)},\\
\exists\,\cp>0
\qquad
\forall\,\varphi&\in\hH^{1}(\om)=\H^{1}(\om)\cap\reals^{\bot}
\qquad
&
\norm{\varphi}_{\Lt(\om)}
&\leq\cp\norm{\na\varphi}_{\Lt(\om)}.
\end{aligned}
\end{align}
With \eqref{eq:introfpest1}
$\naga$ and $\naes$ are well-defined, i.e.,
densely defined and closed linear operators,
with closed ranges,
and so are their Hilbert space adjoints
(maximal and minimal divergence)
\begin{align*}
-\dives=(\naga)^{*}:
\H(\div,\om)\subset\Lt(\om)&\to\Lt(\om);
&
\Phi&\mapsto\div\Phi,\\
-\div_{\ga}=(\naes)^{*}:\Hc(\div,\om)\subset\Lt(\om)&\to\Lt(\om).
\end{align*}
Then the Dirichlet and Neumann negative Laplacians 
\begin{align*}
\Ld=\naga^{*}\naga=-\dives\naga:
D(\Ld)\subset\Lt(\om)&\to\Lt(\om);
&
\phi&\mapsto-\div\na\phi,\\
\Ln=\naes^{*}\naes=-\divga\naes:
D(\Ln)\subset\Lt(\om)&\to\Lt(\om)
\end{align*}
with domains of definition
\begin{align*}
D(\Ld)
&=\big\{\varphi\in\Hc^{1}(\om):\na\varphi\in\H(\div,\om)\big\}
=\big\{\varphi\in\Hc^{1}(\om):\Delta\varphi\in\Lt(\om)\big\},\\
D(\Ln)&=\big\{\varphi\in\H^{1}(\om):\na\varphi\in\Hc(\div,\om)\big\},
\end{align*}
are selfadjoint and non-negative
with kernels $N(\Ld)=\{0\}$ and $N(\Ln)=\reals$. 
With the closed ranges we have the Fredholm alternatives
$$R(\Ld)=N(\Ld)^{\bot}=\Lt(\om),\qquad
R(\Ln)=N(\Ln)^{\bot}=\Lt(\om)\cap\reals^{\bot}=\hLt(\om).$$
Thus $\cLd=\Ld$
and $\cLn=\Ln|_{\reals^{\bot}}$ are bijective
and the inverse operators 
$$\cLd^{-1}:\Lt(\om)\to D(\cLd),\qquad
\cLn^{-1}:\hLt(\om)\to D(\Ln)\cap\reals^{\bot}=D(\cLn)$$
are bounded by \eqref{eq:introfpest1}.

We may also consider over- and underdetermined negative Laplacians $\Lc$ and $\L$,
respectively, cf.~Section \ref{sec:odetlap}.
Now \eqref{eq:introfpest1}, cf.~Lemma \ref{lem:friedrichsH2} and Lemma \ref{lem:fatoolbox1},
yields that 
\begin{align*}
\Lc:\Hc^{2}(\om)\subset\Lt(\om)&\to\Lt(\om);
&
\varphi&\mapsto-\Delta\varphi,\\
\L:=\Lc^{*}:\H(\Delta,\om)\subset\Lt(\om)&\to\Lt(\om)
\end{align*}
are densely defined and closed with kernels
$N(\Lc)=\{0\}$ and $N(\L)=\bbH$
and closed ranges
$$R(\Lc)=N(\L)^{\bot}
=\Lt(\om)\cap\bbH^{\bot}
=\tLt(\om),\qquad
R(\L)=N(\Lc)^{\bot}
=\Lt(\om).$$
Therefore, $\cLc=\Lc$
and $\cL=\L|_{\bbH^{\bot}}$ are bijective
and the inverse operators 
$$\cLc{}^{-1}:\tLt(\om)\to D(\cLc),\qquad
\cL^{-1}:\Lt(\om)\to D(\cL)=D(\L)\cap\bbH^{\bot}=\tH(\Delta,\om)$$
are bounded by \eqref{eq:introfpest1}.

One could say that the dual (adjoint) pair
$\Lc$ and $\L$ defines the minimal and the maximal
$\Lt(\om)$-realisations of the negative Laplacian, in the sense
$$\Lc\subset\Ld,\Ln\subset\L$$
with the two selfadjoint realisations $\Ld$ and $\Ln$ in-between.

Let $f\in\Lt(\om)$,
$\wh{f}\in\hLt(\om)$,
and 
$\wt{f}\in\tLt(\om)$.
Then:
\begin{itemize}
\item
$\mathring{u}=\cLc{}^{-1}\wt{f}\in D(\cLc)$ 
is the unique solution of the overdetermined negative Laplace boundary value problem 
$\cLc\mathring{u}=\wt{f}$.
\item
$u_{\ttD}=\cLd^{-1}f\in D(\cLd)$ 
is the unique solution of the Dirichlet negative Laplace boundary value problem
$\cLd u_{\ttD}=f$.
\item
$u_{\ttN}=\cLn^{-1}\wh{f}\in D(\cLn)$ 
is the unique solution of the Neumann negative Laplace boundary value problem
$\cLn u_{\ttN}=\wh{f}$.
\item
$u=\cL^{-1}f\in D(\cL)$ 
is the unique solution of the underdetermined negative Laplace boundary value problem 
$\cL u=f$.
\end{itemize}

In classical terms we have
\begin{align*}
-\Delta\mathring{u}&=\wt{f},
&
-\Delta u_{\ttD}&=f,
&
-\Delta u_{\ttN}&=\wh{f},
&
-\Delta u&=f
&
&\text{in }\om,\\
\mathring{u}&=0,
&
u_{\ttD}&=0
&
&
&
&
&
&\text{on }\ga,\\
n\cdot\na\mathring{u}&=0,
&
&
&
n\cdot\na u_{\ttN}&=0
&
&
&
&\text{on }\ga,\\
&
&
&
&
u_{\ttN}&\;\bot\;\reals,
&
u&\;\bot\;\bbH.
\end{align*}

To find $u_{\ttD}\in D(\cLd)\subset\Hc^{1}(\om)$ 
and $u_{\ttN}\in D(\cLn)\subset\hH^{1}(\om)$
by variational methods one may consider
the well-known text book formulations
\begin{align*}
\forall\,\varphi&\in\Hc^{1}(\om)
&
\scp{\na u_{\ttD}}{\na\varphi}_{\Lt(\om)}
&=\scp{f}{\varphi}_{\Lt(\om)},\\
\forall\,\psi&\in\H^{1}(\om)\;\big(\text{or }\hH^{1}(\om)\big)
&
\scp{\na u_{\ttN}}{\na\psi}_{\Lt(\om)}
&=\scp{\wh{f}}{\psi}_{\Lt(\om)}.
\end{align*}
In Remark \ref{rem:biharmoverdetlap}
we suggest variational methods
to find the two solutions $\mathring{u}\in D(\cLc)=\Hc^{2}(\om)$ 
and $u\in D(\cL)=\tH(\Delta,\om)$.

Let us go back to the different biharmonic operators
from the beginning of Section \ref{sec:introbilap},
cf.~Section \ref{sec:biharm1} and Section \ref{sec:zoo1}.
From our operator theoretical point of view,
these six different boundary value problems for the biharmonic equation,
are given by the following table:
\begin{align}
\begin{aligned}
\label{eq:bcdef2}
\text{Dirichlet:}&
\quad\eqref{eq:bcD}\wedge\eqref{eq:bcN}
&&\text{given by }\Bd=\L\Lc
&&\text{with bd right inverse }\cBd^{-1}=\cLc{}^{-1}\cL^{-1}\\
\text{Neumann:}&
\quad\eqref{eq:bcDD}\wedge\eqref{eq:bcND}
&&\text{given by }\Bn=\Lc\L
&&\text{with bd right inverse }\cBn^{-1}=\cL^{-1}\cLc{}^{-1}\\
\text{Navier:}&
\quad\eqref{eq:bcD}\wedge\eqref{eq:bcDD}
&&\text{given by }\Bdd=\Ld\Ld
&&\text{with bd right inverse }\cBdd^{-1}=\cLd^{-1}\cLd^{-1}\\
\text{Riquier:}&
\quad\eqref{eq:bcN}\wedge\eqref{eq:bcND}
&&\text{given by }\Bnn=\Ln\Ln
&&\text{with bd right inverse }\cBnn^{-1}=\cLn^{-1}\cLn^{-1}\\
\text{N-D:}&
\quad\eqref{eq:bcD}\wedge\eqref{eq:bcND}
&&\text{given by }\Bnd=\Ln\Ld
&&\text{with bd right inverse }\cB_{\ttN,\ttD}^{-1}=\cLd^{-1}\cLn^{-1}\\
\text{D-N:}&
\quad\eqref{eq:bcN}\wedge\eqref{eq:bcDD}
&&\text{given by }\Bdn=\Ld\Ln
\end{aligned}
\end{align}

By our theory the inverses of the first five biharmonic operators are well-defined,
while the last inverse of the biharmonic operator $\Bdn$ is not in terms of $\cLn^{-1}\cLd^{-1}$.
Since 
$$\Lc{}^{*}=\L,\qquad
\L^{*}=\Lc,\qquad
\Ld^{*}=\Ld,\qquad
\Ln^{*}=\Ln,$$
we see that the operators corresponding to 
the Dirichlet, Neumann, Navier, and Riquier biharmonic problems
are selfadjoint (and non-negative), while the others are not.

Let $f\in\Lt(\om)$,
$\wh{f}\in\hLt(\om)$,
and 
$\wt{f}\in\tLt(\om)$.
Then:
\begin{itemize}
\item
$u_{\ttD}=\cBd^{-1}=\cLc^{-1}\cL^{-1}f\in D(\cBd)$ 
is the unique solution of the Dirichlet biharmonic boundary value problem
$\cBd u_{\ttD}=f$.
\item
$u_{\ttD\ttD}=\cBdd^{-1}=\cLd^{-1}\cLd^{-1}f\in D(\cBdd)$ 
is the unique solution of the Navier biharmonic boundary value problem
$\cBdd u_{\ttD\ttD}=f$.
\item
$u_{\ttN}=\cBn^{-1}=\cL^{-1}\cLc^{-1}\wt{f}\in D(\cBn)$ 
is the unique solution of the Neumann biharmonic boundary value problem
$\cBn u_{\ttN}=\wt{f}$.
\item
$u_{\ttN\ttN}=\cBnn^{-1}=\cLn^{-1}\cLn^{-1}\wh{f}\in D(\cBnn)$ 
is the unique solution of the Riquier biharmonic boundary value problem
$\cBnn u_{\ttN\ttN}=\wh{f}$.
\item
$u_{\ttN\ttD}=\cBnd^{-1}=\cLd^{-1}\cLn^{-1}\wh{f}\in D(\cBnd)$ 
is the unique solution of the N-D biharmonic boundary value problem
$\cBnd u_{\ttN\ttD}=\wh{f}$.
\end{itemize}

In classical terms we have
\begin{align*}
\Delta^{2}u_{\ttD}&=f,
&
\Delta^{2}u_{\ttD\ttD}&=f,
&
\Delta^{2}u_{\ttN}&=\wt{f},
&
\Delta^{2}u_{\ttN\ttN}&=\wh{f},
&
\Delta^{2}u_{\ttN\ttD}&=\wh{f}
&
&\text{in }\om,\\
u_{\ttD}&=0,
&
u_{\ttD\ttD}&=0,
&
&
&
&
&u_{\ttN\ttD}&=0
&
&\text{on }\ga,\\
n\cdot\na u_{\ttD}&=0,
&
&
&
&
&
n\cdot\na u_{\ttN\ttN}&=0
&
&
&
&\text{on }\ga,\\
&
&
\Delta u_{\ttD\ttD}&=0,
&
\Delta u_{\ttN}&=0
&
&
&
&
&
&\text{on }\ga,\\
&
&
&
&
n\cdot\na\Delta u_{\ttN}&=0,
&
n\cdot\na\Delta u_{\ttN\ttN}&=0,
&
n\cdot\na\Delta u_{\ttN\ttD}&=0
&
&\text{on }\ga,\\
&
&
&
&
u_{\ttN}&\;\bot\;\bbH,
&
u_{\ttN\ttN}&\;\bot\;\reals,
&
&\\
&
&
&
&
&
&
\Delta u_{\ttN\ttN}&\;\bot\;\reals,
&
\Delta u_{\ttN\ttD}&\;\bot\;\reals.
\end{align*}

In fact, in Section \ref{sec:zoo1} and Section \ref{sec:odetbiharm} 
we present \emph{eighteen} different biharmonic operators,
and it turns out that \emph{thirteen} of those are well defined 
and lead to uniquely solvable boundary value problems for the bi-Laplacian.
Our general theory extends also to mixed boundary conditions 
giving even more well defined bi-Laplacians,
cf.~Section \ref{sec:mixedbc1}.

\begin{rem}
\label{rem:intronoreg1}
Concerning \eqref{eq:introestDelta}, \eqref{eq:introcpt}, and \eqref{eq:introX}
we can show remarkable results for all bounded domains $\om$.
Note that no(!) regularity assumptions on $\om$ are needed at all.
\begin{itemize}
\item[\bf(i)]
In Theorem \ref{theo:cptharm} we prove that indeed
$$D(\cL)=\tH(\Delta,\om)=\H(\Delta,\om)\cap\bbH^{\bot}\hookrightarrow\Lt(\om)$$
is compact. 
\item[\bf(ii)]
Moreover, for all 
$\varphi\in D(\cLc)=\Hc^{2}(\om)$
and all $\varphi\in D(\cL)$
we have
$$\norm{\varphi}_{\Lt(\om)}
\leq\cf^{2}\norm{\Delta\varphi}_{\Lt(\om)},$$
cf.~Remark \ref{rem:fpest}.
\item[\bf(iii)]
For the variational space we see 
\begin{align*}
\tH(\Delta,\om)
=\cLc\cLc{}^{-1}\tH(\Delta,\om)
&=\Delta\big\{\varphi\in\Hc^{2}(\om):\Delta\varphi\in D(\cL)\big\}\\
&=\Delta\big\{\varphi\in\Hc^{2}(\om):\Delta^{2}\varphi\in\Lt(\om)\big\},
\end{align*}
which -- in case of a $\C^{4}$-boundary $\ga$ --
equals the space $\Delta\big(\H^{4}(\om)\cap\Hc^{2}(\om)\big)$
by standard regularity theory for the Dirichlet bi-Laplacian.
\end{itemize}
\end{rem}

\begin{rem}
\label{rem:intronoreg2}
Analogously we have:
\begin{itemize}
\item[\bf(i)]
Theorem \ref{theo:cptbiharm} shows that 
$$D(\cB)=\cH(\Delta^{2},\om)=\H(\Delta^{2},\om)\cap\bbBH^{\bot}\hookrightarrow\Lt(\om)$$
is compact.
\item[\bf(ii)]
For all 
$\varphi\in D(\cBc)=\Hc^{4}(\om)$
and all $\varphi\in D(\cB)$
we have
$$\norm{\varphi}_{\Lt(\om)}
\leq\cf^{4}\norm{\Delta^{2}\varphi}_{\Lt(\om)},$$
cf.~\eqref{eq:est1bilap}.
\item[\bf(iii)]
For the variational space we have
\begin{align*}
\cH(\Delta^{2},\om)
=\cBc\cBc{}^{-1}\cH(\Delta^{2},\om)
&=\Delta^{2}\big\{\varphi\in\Hc^{4}(\om):\Delta^{2}\varphi\in D(\cB)\big\}\\
&=\Delta^{2}\big\{\varphi\in\Hc^{4}(\om):\Delta^{4}\varphi\in\Lt(\om)\big\},
\end{align*}
which -- in case of a $\C^{8}$-boundary $\ga$ --
equals the space $\Delta^{2}\big(\H^{8}(\om)\cap\Hc^{4}(\om)\big)$
by standard regularity theory for higher order elliptic Dirichlet problems.
\end{itemize}
\end{rem}

\subsubsection{Exchanging Boundary Conditions by Laplacians}
\label{sec:introtrick}

Using the latter four different Laplacians,
cf.~\eqref{eq:cLinv},
we can solve the 
Neumann biharmonic problem by the Dirichlet biharmonic problem
and vice versa. For example, 
$$\cBn^{-1}
=\cLc\cLc{}^{-1}\cL^{-1}\cLc{}^{-1}
=\cLc\cBd^{-1}\cLc{}^{-1},$$
i.e., the solution $u=\cBn^{-1}f$ of the Neumann biharmonic problem
can be found by solving for the overdetermined Laplacian $\cLc{}^{-1}f$,
then for the Dirichlet biharmonic problem $v=\cBd^{-1}\cLc{}^{-1}f$,
and finally taking $u=\Delta v$.
Analogously, we have 
$$\cBd^{-1}
=\cL\cL^{-1}\cLc{}^{-1}\cL^{-1}
=\cL\cBn^{-1}\cLc{}^{-1}.$$

This ``trick'' can extended to a lot more ``allowed'' combinations, such as
$$\cB_{\ttN,\ttD}^{-1}
=\cLd\cLd^{-1}\cLd^{-1}\cLn^{-1}
=\cLd\cBdd^{-1}\cLn^{-1},$$
just to mention one of many options.

\subsection{Alternative Approach by the Hessian Complex}
\label{sec:introalt}

For $\phi\in\C^{\infty}(\rt)$ we have point-wise 
\begin{align}
\label{eq:DeltadivDivnana}
\Delta^{2}\phi
=\sum_{i,j}\p_{i}^{2}\p_{j}^{2}\phi
=\sum_{i,j}\p_{i}\p_{j}\p_{i}\p_{j}\phi
=\divDiv\nana\phi,
\end{align}
which then extends also to distributions $\phi$.
Here, $\nana\phi$ denotes the Hessian of $\phi$
and $\Div$ acts as row-wise incarnation of $\div$.
Note that $\divDiv$ is the formal adjoint of $\nana$.
In the following we use again the basic concepts of functional analysis, 
cf.~Section \ref{sec:tinyfatoolbox}.

By \eqref{eq:DeltadivDivnana}
another way to look at the biharmonic equation,
which respects more the underlining geometry (Hessian Hilbert complex)
of the problem and the corresponding operators,
is to investigate the two Hessian operators
($\bbS$ indicates symmetric tensors)
\begin{align*}
\nanaga:
\Hc^{2}(\om)\subset\Lt(\om)&\to\Lt_{\bbS}(\om);
&
\phi&\mapsto\nana\phi,\\
\nanaes:\H^{2}(\om)\subset\Lt(\om)&\to\Lt_{\bbS}(\om),
\end{align*}
cf.~Section \ref{sec:appbiharm} and \cite{PZ2020a,AH2021a,PS2024a}
for more details and results on the Hessian complex.
Note that Friedrichs' estimate \eqref{eq:introfpest1} shows
\begin{align}
\label{eq:introfest1}
\exists\,\cf>0\quad
\forall\,\varphi\in\Hc^{2}(\om)\qquad
\norm{\varphi}_{\Lt(\om)}
\leq\cf\norm{\na\varphi}_{\Lt(\om)}
\leq\cf^{2}\norm{\nana\varphi}_{\Lt(\om)}.
\end{align}
Moreover, $\varphi\in\Lt(\om)$ with $\nana\varphi\in\Lt(\om)$
implies $\varphi\in\H^{2}(\om)$ by the Ne\v{c}as/Lions lemma,
and we have Poincar\'e's estimate
\begin{align}
\label{eq:intropest1}
\exists\,\cp>0\quad
\forall\,\varphi\in\hhH^{2}(\om)=\H^{2}(\om)\cap\P_{1}^{\bot}\qquad
\norm{\varphi}_{\Lt(\om)}
\leq\cp\norm{\na\varphi}_{\Lt(\om)}
\leq\cp^{2}\norm{\nana\varphi}_{\Lt(\om)},
\end{align}
where $\P_{1}$ denotes the first order polynomials.
With \eqref{eq:introfest1}, Ne\v{c}as' lemma, and \eqref{eq:intropest1}
$\nanaga$ and $\nanaes$ are well-defined, i.e.,
densely defined and closed linear operators.
Their Hilbert space adjoints are given by 
\begin{align*}
\divDiv_{\bbS,\emptyset}=(\nanaga)^{*}:
\H_{\bbS}(\divDiv,\om)\subset\Lt_{\bbS}(\om)&\to\Lt(\om);
&
\Phi&\mapsto\divDiv\Phi,\\
\divDiv_{\bbS,\ga}=(\nanaes)^{*}:\Hc_{\bbS}(\divDiv,\om)\subset\Lt_{\bbS}(\om)&\to\Lt(\om),
\end{align*}
where $\H_{\bbS}(\divDiv,\om)=\big\{\Phi\in\Lt_{\bbS}(\om):\divDiv\Phi\in\Lt(\om)\big\}$
and $\Hc_{\bbS}(\divDiv,\om)$ denotes the respective closure of symmetric test fields
in the graph norm.

We can then investigate the biharmonic operators 
\begin{align*}
\B_{\ga}=\nanaga^{*}\nanaga=\divDiv_{\bbS,\emptyset}\nanaga:
D(\B_{\ga})\subset\Lt(\om)&\to\Lt(\om);
&
\phi&\mapsto\divDiv\nana\phi,\\
\B_{\emptyset}=\nanaes^{*}\nanaes=\divDiv_{\bbS,\ga}\nanaes:
D(\B_{\emptyset})\subset\Lt(\om)&\to\Lt(\om),
\end{align*}
which may be called Dirichlet resp. Neumann biharmonic operator as well.
Note that indeed by Remark \ref{rem:H2bc}
the ``new'' Dirichlet biharmonic operator equals the ``old'' one from above, i.e.,
$$\B_{\ga}=\Bd=\L\Lc,$$
but the ``new'' Neumann biharmonic operator $\B_{\emptyset}$
is not a biharmonic operator with any combination of the boundary conditions 
\eqref{eq:bcD}-\eqref{eq:bcND}.
In fact, while the boundary conditions of $\Bn=\Lc\L$
are imposed on the scalar $\Delta u$,
the boundary conditions of $\B_{\emptyset}$
are imposed on the symmetric tensor $S=\nana u$. 

$\B_{\ga}$ and $\B_{\emptyset}$ are selfadjoint and non-negative
with kernels $N(\B_{\ga})=\{0\}$ and $N(\B_{\emptyset})=\P_{1}$. 
By \eqref{eq:introfest1} and \eqref{eq:intropest1}
the ranges are closed and we have the Fredholm alternatives
$$R(\B_{\ga})=N(\B_{\ga})^{\bot}=\Lt(\om),\qquad
R(\B_{\emptyset})=N(\B_{\emptyset})^{\bot}=\Lt(\om)\cap\P_{1}^{\bot}=:\hhLt(\om).$$
Therefore, $\cB_{\ga}=\B_{\ga}$
and $\cB_{\emptyset}=\B_{\emptyset}|_{\P_{1}^{\bot}}$ are bijective
and the inverse operators 
$$\cB_{\ga}^{-1}:\Lt(\om)\to D(\cB_{\ga}),\qquad
\cB_{\emptyset}^{-1}:\hhLt(\om)\to D(\B_{\emptyset})\cap\P_{1}^{\bot}=D(\cB_{\emptyset})$$
are bounded, i.e., $\norm{\cB_{\ga}^{-1}}\leq\cf^{4}$ and $\norm{\cB_{\emptyset}^{-1}}\leq\cp^{4}$.

Let $f\in\Lt(\om)$ and $f_{1}\in\hhLt(\om)$.
Then $u_{\ga}=\cB_{\ga}^{-1}f\in D(\cB_{\ga})$ 
is the unique solution of
the Dirichlet biharmonic boundary value problem
$\cB_{\ga}u_{\ga}=f$, and 
$u_{\emptyset}=\cB_{\emptyset}^{-1}f_{1}\in D(\cB_{\emptyset})$ 
is the unique solution of
the Neumann biharmonic boundary value problem 
$\cB_{\emptyset}u_{\emptyset}=f_{1}$.
In classical terms we have
\begin{align*}
\divDiv\nana u_{\ga}&=f,
&
\divDiv\nana u_{\emptyset}&=f_{1}
&
&\text{in }\om,\\
u_{\ga}&=0,
&
(\nana u_{\emptyset})\,n&=0
&
&\text{on }\ga,\\
\na u_{\ga}&=0,
&
(\Div\nana u_{\emptyset})\cdot n&=0
&
&\text{on }\ga,\\
&&
u_{\emptyset}&\;\bot\;\P_{1}.
\end{align*}

To find $u\in D(\cB_{\ga})\subset\Hc^{2}(\om)$ 
and $v\in D(\cB_{\emptyset})\subset\hhH^{2}(\om):=\H^{2}(\om)\cap\P_{1}^{\bot}$
by variational methods one may consider
\begin{align*}
\forall\,\varphi&\in\Hc^{2}(\om)
&
\scp{\nana u}{\nana\varphi}_{\Lt_{\bbS}(\om)}
&=\scp{f}{\varphi}_{\Lt(\om)},\\
\forall\,\psi&\in\H^{2}(\om)\;\big(\text{or }\hhH^{2}(\om)\big)
&
\scp{\nana v}{\nana\psi}_{\Lt_{\bbS}(\om)}
&=\scp{g}{\psi}_{\Lt(\om)}.
\end{align*}

It is worth noting that the results in \cite[Section 4]{PZ2020a} show that 
the Dirichlet biharmonic problem
$$\cB_{\ga}u=\divDiv_{\bbS,\emptyset}\nanaga u=f$$
splits up into a sequence of \emph{three second order} (elliptic) boundary value problems 
indicated by the matrix representation
$$\begin{bmatrix}
\Ld & \tr\symRot_{\bbT,\emptyset} & 3\\
0 & \Rot_{\bbS,\ga}\symRot_{\bbT,\emptyset} & \Rot_{\bbS,\ga}\tr^{*}\\
0 & 0 & \Ld
\end{bmatrix}
\begin{bmatrix}
u\\
E\\
v
\end{bmatrix}
=\begin{bmatrix}
0\\
0\\
f
\end{bmatrix},$$
cf.~Section \ref{sec:appbiharm} for definitions.
The matrix has the structure 
$$\begin{bmatrix}
a^{*}a & c & 3\\
0 & b^{*}b & c^{*}\\
0 & 0 & a^{*}a
\end{bmatrix}$$
with $a=\A_{0}=\nac$ taken from the de Rham complex \eqref{complex:derham1},
and $b=\A_{1}^{*}=\symRot_{\bbT,\emptyset}$ and $c=\tr b$ 
from the Hessian complex \eqref{complex:biharm}.

\section{Operator Theory for Biharmonic Equations}
\label{sec:optheo}

Let us begin with some abstract basics.

\subsection{Tiny FA-ToolBox}
\label{sec:tinyfatoolbox}

We recall a few results from linear functional analysis.
In particular, we use fundamental results from the so-called FA-ToolBox, see, e.g.,
 \cite{P2019a,P2020a}, cf.~\cite{P2017a,P2019b,PS2022a,PS2022b,PS2024a,PZ2020a,PZ2023a}.

\subsubsection{Single Operators}
\label{sec:tinyfatoolbox1}

Let us consider a densely defined and closed linear operator $\A$
between two Hilbert spaces $\H_{0}$ and $\H_{1}$ 
together with its (densely defined and closed) Hilbert space adjoint $\A^{*}$, that is
$$\A:D(\A)\subset\H_{0}\to\H_{1},\qquad
\A^{*}:D(\A^{*})\subset\H_{1}\to\H_{0}.$$
In general, $\A$ and $\A^{*}$ are unbounded
and characterised by
$$\forall\,x\in D(\A)\quad
\forall\,y\in D(\A^{*})\qquad
\scp{\A x}{y}_{\H_{1}}
=\scp{x}{\A^{*}y}_{\H_{0}}.$$
Since $\A^{**}=\A$ we call $(\A,\A^{*})$ a dual pair.
Note that $R(\A)$ is closed if and only if
$R(\A^{*})$ is closed by the closed range theorem.
Moreover, the projection theorem 
yields the orthogonal decompositions (Helmholtz type decompositions)
\begin{align}
\label{eq:helm1}
\H_{0}=\overline{R(\A^{*})}\oplus_{\H_{0}}N(\A),\qquad
\H_{1}=\overline{R(\A)}\oplus_{\H_{1}}N(\A^{*}),
\end{align}
which suggest to investigate the injective restrictions,
also called reduced operators and denoted by calligraphical letters,
$$\cA:=\A|_{N(\A)^{\bot}},\qquad
\cA^{*}:=\A^{*}|_{N(\A^{*})^{\bot}},$$
more precisely,
\begin{align*}
\cA:D(\cA)\subset N(\A)^{\bot}&\to\overline{R(\A)}=N(\A^{*})^{\bot},
&
D(\cA)&:=D(\A)\cap N(\A)^{\bot},\\
\cA^{*}:D(\cA^{*})\subset N(\A^{*})^{\bot}&\to\overline{R(\A^{*})}=N(\A)^{\bot},
&
D(\cA^{*})&:=D(\A^{*})\cap N(\A^{*})^{\bot}.
\end{align*}
$(\cA,\cA^{*})$ are also densely defined and closed
forming another dual pair.
Moreover, by \eqref{eq:helm1} we have
\begin{align}
\label{eq:helmrange}
\begin{aligned}
D(\A)&=D(\cA)\oplus_{\H_{0}}N(\A),
&
D(\A^{*})&=D(\cA^{*})\oplus_{\H_{1}}N(\A^{*}),\\
R(\A)&=R(\cA),
&
R(\A^{*})&=R(\cA^{*}).
\end{aligned}
\end{align}
Here we have used the symbols
$\overline{\,\cdot\,}$, $\oplus$, and $\bot$
for the closure, the orthogonal sum, and the orthogonal complement, respectively. 

From \cite[Lemma 4.1, Remark 4.2]{P2019b}, see also 
 \cite[Lemma 2.1, Lemma 2.2]{P2020a} or \cite[Lemma 2.1, Lemma 2.4]{PV2020a},
we cite the following elementary result.

\begin{lem}[fundamental FA-ToolBox lemma 1]
\label{lem:fatoolbox1}
The following assertions are equivalent:
\begin{itemize}
\item[\bf(i)]
$\exists\,\ca>0\quad\forall\,x\in D(\cA)\qquad\norm{x}_{\H_{0}}\leq\ca\norm{\A x}_{\H_{1}}$
\item[\bf(i$^{*}$)]
$\exists\,c_{\A^{*}}>0\quad\forall\,y\in D(\cA^{*})\qquad\norm{y}_{\H_{1}}\leq c_{\A^{*}}\norm{\A^{*}y}_{\H_{0}}$
\item[\bf(ii)]
$R(\A)$ is closed.
\item[\bf(ii$^{*}$)]
$R(\A^{*})$ is closed.
\item[\bf(iii)]
$\cA^{-1}:R(\A)\to D(\cA)$ is bounded.
\item[\bf(iii$^{*}$)]
$(\cA^{*})^{-1}:R(\A^{*})\to D(\cA^{*})$ is bounded.
\end{itemize}
All these assertions hold,
if the embedding $D(\cA)\hookrightarrow H_{0}$ is compact.
\begin{itemize}
\item[\bf(iv)]
$D(\cA)\hookrightarrow H_{0}$ is compact,
if and only if $D(\cA^{*})\hookrightarrow H_{1}$ is compact.
\end{itemize}
\end{lem}

\begin{rem}[fundamental FA-ToolBox lemma 1]
\label{rem:fatoolbox1}
If the estimate in (i) holds with $\ca$ then (ii) holds with $c_{\A^{*}}=\ca$, and vice versa. 
For the best constants we have
$$\bnorm{\cA^{-1}}_{R(\A)\to\H_{0}}
=\ca
=c_{\A^{*}}
=\bnorm{(\cA^{*})^{-1}}_{R(\A^{*})\to\H_{1}}.$$
Lemma \ref{lem:fatoolbox1} shows that the key point 
to a proper solution theory in the sense of Hadamard
is a close range (ii) or, equivalently, a Friedrichs/Poincar\'e type estimate (i).
\end{rem}

\begin{lem}[automatic regularity]
\label{lem:fatoolboxreg}
$\A^{*}\A$, $\A\A^{*}$ and $\cA^{*}\cA$, $\cA\cA^{*}$
are selfadjoint and nonnegative.
$\cA^{*}\cA=\A^{*}\cA$ and $\cA\cA^{*}=\A\cA^{*}$ are positive (and hence injective).
Moreover, the automatic regularity \eqref{eq:helmrange} extends to
\begin{align*}
\ldots=R(\A\A^{*}\A\dots)=R(\A\A^{*})=R(\A)&=R(\cA)=R(\cA\cA^{*})=R(\cA\cA^{*}\cA\dots)=\ldots,\\
\ldots=R(\A^{*}\A\A^{*}\dots)=R(\A^{*}\A)=R(\A^{*})&=R(\cA^{*})=R(\cA^{*}\cA)=R(\cA^{*}\cA\cA^{*}\dots)=\ldots
\end{align*}
and it holds
$$N(\A)=N(\A^{*}\A)=N(\A\A^{*}\A)=\ldots,\qquad
N(\A^{*})=N(\A\A^{*})=N(\A^{*}\A\A^{*})=\ldots.$$
\end{lem}

\begin{lem}[fundamental FA-ToolBox lemma 2]
\label{lem:fatoolbox2}
Let $R(\A)$ be closed.

\begin{itemize}
\item[\bf(i)]
$\cA^{*}\cA:D(\cA^{*}\cA)\subset N(\A)^{\bot}\to R(\A^{*})=N(\A)^{\bot}$
is bijective with bounded inverse
$$(\cA^{*}\cA)^{-1}=\cA^{-1}(\cA^{*})^{-1}:R(\A^{*})\to D(\cA^{*}\cA),\qquad
\bnorm{(\cA^{*}\cA)^{-1}}_{R(\A^{*})\to\H_{0}}=\ca^{2}.$$
For $x\in D(\cA^{*}\cA)$ it holds
\begin{align}
\label{eq:estAA}
\norm{x}_{\H_{0}}\leq\ca\norm{\A x}_{\H_{1}}\leq\ca^{2}\norm{\A^{*}\A x}_{\H_{0}}.
\end{align}
Moreover, $\cA^{*}\cA=\A^{*}\cA$.
\item[\bf(i$^{*}$)]
Interchanging $\A$ and $\A^{*}$ we get a similar results for $\cA\cA^{*}$.
\end{itemize}

Solving $\cA^{*}\cA x=f$:

\begin{itemize}
\item[\bf(ii)]
For $f\in R(\A^{*})$ the unique solution $x:=(\cA^{*}\cA)^{-1}f\in D(\cA^{*}\cA)$ 
of $\cA^{*}\cA x=f$ can be found by the variational formulation
\begin{align*}
\forall\,\phi\in D(\cA)\qquad
\scp{\A x}{\A\phi}_{\H_{1}}
=\scp{f}{\phi}_{\H_{0}},
\end{align*}
which holds also for all $\phi\in D(\A)$ as
\begin{align}
\label{eq:varform2}
\scp{\A x}{\A\phi}_{\H_{1}}
=\scp{\A x}{\A\pi_{\A^{*}}\phi}_{\H_{1}}
=\scp{f}{\pi_{\A^{*}}\phi}_{\H_{0}},
=\scp{\pi_{\A^{*}}f}{\phi}_{\H_{0}},
=\scp{f}{\phi}_{\H_{0}},
\end{align}
where $\pi_{\A^{*}}$ denotes the orthogonal projector onto $N(\A)^{\bot}=R(\A^{*})$,
which implies also $\pi_{\A^{*}}:D(\A)\to D(\cA)$.
\item[\bf(ii$^{*}$)]
Interchanging $\A$ and $\A^{*}$ we get a similar variational formulation for $\cA\cA^{*}y=g$.
\end{itemize}

Solving $\cA x=g$:

\begin{itemize}
\item[\bf(iii)]
For $g\in R(\A)$ the unique solution $x:=\cA^{-1}g\in D(\cA)$ 
of $\cA x=g$ can be found by the variational formulation
\begin{align*}
\forall\,\phi\in D(\cA)\qquad
\scp{\A x}{\A\phi}_{\H_{1}}
=\scp{g}{\A\phi}_{\H_{1}},
\end{align*}
which holds also for all $\phi\in D(\A)$ since $R(\cA)=R(\A)$.
Note that $\A x-g\in R(\A)$ belongs to $N(\cA^{*})=\{0\}$.

Another way is to use a potential $y$ with $\A^{*}y=x$, i.e.,
to compute the unique potential 
$y:=(\cA^{*})^{-1}x=(\cA^{*})^{-1}\cA^{-1}g\in D(\cA\cA^{*})$
which satisfies $\cA\cA^{*}y=g$. By (ii) we can find $y\in D(\cA^{*})$
by the variational formation 
\begin{align*}
\forall\,\psi\in D(\A^{*})\qquad
\scp{\A^{*}y}{\A^{*}\psi}_{\H_{0}}
=\scp{g}{\psi}_{\H_{1}}.
\end{align*}
\item[\bf(iii$^{*}$)]
Interchanging $\A$ and $\A^{*}$ we get similar variational formulations
for $\cA^{*}y=f$.

\item[\bf(iv)]
Solving the latter variational formulations leads to saddle point problems 
which are tricky to handle.
A comprehensive theory can be found in \cite{P2020a}.
\end{itemize}
\end{lem}

\subsubsection{Helmholtz Projections}
\label{sec:tinyfatoolbox1add}

The projections in \eqref{eq:helm1} and \eqref{eq:helmrange} can be computed as follows:

Let $R(\A)$ be closed and let us consider, e.g., $\H_{1}=R(\A)\oplus_{\H_{1}}N(\A^{*})$.

For $g\in\H_{1}$ the variational formulation in Lemma \ref{lem:fatoolbox2}(iii)
computes the orthogonal projections 
$\pi_{\A}g$ onto $R(\A)$ and
$(1-\pi_{\A})g$ onto $N(\A^{*})=R(\A)^{\bot}$.
More precisely, $x\in D(\A)$ such that 
\begin{align}
\label{eq:projbilin}
\forall\,\phi\in D(\A)\qquad
\scp{\A x}{\A\phi}_{\H_{1}}
=\scp{g}{\A\phi}_{\H_{1}}
\end{align}
implies $\A x-g\in N(\A^{*})$ and thus we get \eqref{eq:helm1}, i.e.,
with $\pi_{\A}g=\A x$
$$g=\A x-\A x+g
=\pi_{\A}g+(1-\pi_{\A})g\in R(\A)\oplus_{\H_{1}}N(\A^{*}).$$
Note that $\A x-g\in N(\cA^{*})=\{0\}$ if and only if $g\in R(\A)$.

Now let us consider the bounded linear operator 
$\A:D(\A)\to\H_{1}$ and its Banach space adjoint $\A'$.
Using the Riesz isometry $\calR_{\H_{1}}:\H_{1}\to\H_{1}'$ we introduce the modified adjoint
$$\A^{\top}:=\A'\calR_{\H_{1}}:\H_{1}\to D(\A)';\qquad
y\mapsto\A'\calR_{\H_{1}}y(\,\cdot\,)=\calR_{\H_{1}}y(\A\cdot\,)=\scp{\A\cdot\,}{y}_{\H_{1}}.$$ 
Then $\A'$ and $\A^{\top}$ are bounded linear operators with
$\norm{\A^{\top}}=\norm{\A'}=\norm{\A}$
and $\A^{\top}$ is an extension of $\A^{*}$.
Moreover, $N(\A^{\top})=R(\A)^{\bot}=N(\A^{*})$
and $R(\A^{\top})=R(\A')$ is closed by the closed range theorem.
Therefore, $\A^{\top}|_{R(\A)}$ is boundedly invertible on $R(\A')$
by the bounded inverse theorem, i.e.,
$$\cA^{\top}:=\A^{\top}|_{R(\A)}:R(\A)\to R(\A');\qquad
y\mapsto\A^{\top}y$$
is a topological isomorphism.

\eqref{eq:projbilin} translates equivalently to 
$$\A^{\top}g=\A^{\top}\A x=\cA^{\top}\A x.$$
Hence
$$\pi_{\A}g
=\A x
=(\cA^{\top})^{-1}\A^{\top}g
=\A\cA^{-1}(\cA^{\top})^{-1}\A^{\top}g
=\A(\cA^{\top}\cA)^{-1}\A^{\top}g.$$
Note that we have indeed 
$\pi_{\A}^{2}
=\A(\cA^{\top}\cA)^{-1}\A^{\top}\A(\cA^{\top}\cA)^{-1}\A^{\top}
=\A(\cA^{\top}\cA)^{-1}\A^{\top}
=\pi_{\A}$
and that for $g\in R(\A)$ it holds
$\pi_{\A}g=(\cA^{\top})^{-1}\A^{\top}g=(\cA^{\top})^{-1}\cA^{\top}g=g$.
Finally, 
$$\pi_{\A}=\A(\cA^{\top}\cA)^{-1}\A^{\top}
=(\cA^{\top})^{-1}\A^{\top}:\H_{1}\to R(\A),$$
where $\A^{\dagger}:=(\cA^{\top}\cA)^{-1}\A^{\top}$ is often called Moore/Penrose inverse of $\A$.
Note that $\A^{\dagger}$ is a bounded right inverse of $\A$, i.e.,
$$\A\A^{\dagger}
=(\cA^{\top})^{-1}\A^{\top},\qquad
\A\A^{\dagger}\big|_{R(\A)}
=(\cA^{\top})^{-1}\cA^{\top}
=\id_{R(\A)}.$$

\subsubsection{Operator Complexes}
\label{sec:tinyfatoolbox2}

Let $\H_{2}$ be another Hilbert space and let
\begin{equation}
\label{complex:A01}
\def\arrowlength{6ex}
\def\arrowdistance{.8}
\begin{tikzcd}[column sep=\arrowlength]
\cdots 
\arrow[r, rightarrow, shift left=\arrowdistance, "\cdots"] 
\arrow[r, leftarrow, shift right=\arrowdistance, "\cdots"']
& 
\H_{0} 
\ar[r, rightarrow, shift left=\arrowdistance, "\A_{0}"] 
\ar[r, leftarrow, shift right=\arrowdistance, "\A_{0}^{*}"']
& 
\H_{1}
\arrow[r, rightarrow, shift left=\arrowdistance, "\A_{1}"] 
\arrow[r, leftarrow, shift right=\arrowdistance, "\A_{1}^{*}"']
& 
\H_{2}
\arrow[r, rightarrow, shift left=\arrowdistance, "\cdots"] 
\arrow[r, leftarrow, shift right=\arrowdistance, "\cdots"']
&
\cdots 
\end{tikzcd}
\end{equation}
be a primal and dual Hilbert complex, i.e.,
\begin{align*}
\A_{0}:D(\A_{0})\subset\H_{0}&\to\H_{1},
&
\A_{1}:D(\A_{1})\subset\H_{1}&\to\H_{2},\\
\A_{0}^{*}:D(\A_{0}^{*})\subset\H_{1}&\to\H_{0},
&
\A_{1}^{*}:D(\A_{1}^{*})\subset\H_{2}&\to\H_{1}
\end{align*}
are densely defined and closed linear operators
satisfying the complex property
\begin{align}
\label{eq:compprop}
\A_{1}\A_{0}\subset0.
\end{align}
Note that \eqref{eq:compprop} is equivalent to $R(\A_{0})\subset N(\A_{1})$
which is equivalent to $R(\A_{1}^{*})\subset N(\A_{0}^{*})$
(dual complex property) as 
$R(\A_{1}^{*})
\subset\overline{R(\A_{1}^{*})}
=N(\A_{1})^{\bot_{\H_{1}}}
\subset R(\A_{0})^{\bot_{\H_{1}}}
=N(\A_{0}^{*})$
and vice versa. 

Defining the cohomology group 
$$N_{0,1}:=N(\A_{1})\cap N(\A_{0}^{*})$$
we get the following orthogonal Helmholtz-type decompositions,
cf.~\eqref{eq:helm1}.

\begin{lem}[Helmholtz decomposition]
\label{lem:toolboxhelm1}
The orthogonal Helmholtz-type decompositions
\begin{align}
\label{eq:helm2}
\begin{aligned}
\H_{1}&=\overline{R(\A_{0})}\oplus_{\H_{1}}N(\A_{0}^{*}),
&
\H_{1}&=N(\A_{1})\oplus_{\H_{1}}\overline{R(\A_{1}^{*})},\\
N(\A_{1})&=\overline{R(\A_{0})}\oplus_{\H_{1}}N_{0,1},
&
N(\A_{0}^{*})&=N_{0,1}\oplus_{\H_{1}}\overline{R(\A_{1}^{*})},\\
D(\A_{1})&=\overline{R(\A_{0})}\oplus_{\H_{1}}\big(D(\A_{1})\cap N(\A_{0}^{*})\big),
&
D(\A_{0}^{*})&=\big(N(\A_{1})\cap D(\A_{0}^{*})\big)\oplus_{\H_{1}}\overline{R(\A_{1}^{*})},\\
D(\A_{0}^{*})&=D(\cA_{0}^{*})\oplus_{\H_{1}}N(\A_{0}^{*}),
&
D(\A_{1})&=N(\A_{1})\oplus_{\H_{1}}D(\cA_{1}),
\end{aligned}
\end{align}
as well as $R(\cA_{0}^{*})=R(\A_{0}^{*})$ 
and $R(\cA_{1})=R(\A_{1})$ hold. Moreover,
\begin{align*}
\H_{1}
&=\overline{R(\A_{0})}\oplus_{\H_{1}}N_{0,1}\oplus_{\H_{1}}\overline{R(\A_{1}^{*})},\\
D(\A_{0}^{*})
&=D(\cA_{0}^{*})\oplus_{\H_{1}}N_{0,1}\oplus_{\H_{1}}\overline{R(\A_{1}^{*})},\\
D(\A_{1})
&=\overline{R(\A_{0})}\oplus_{\H_{1}}N_{0,1}\oplus_{\H_{1}}D(\cA_{1}),\\
D(\A_{1})\cap D(\A_{0}^{*})
&=D(\cA_{0}^{*})\oplus_{\H_{1}}N_{0,1}\oplus_{\H_{1}}D(\cA_{1}).
\end{align*}
\end{lem}

\subsection{Applications to the De Rham Complex}
\label{sec:appderham}

For Lebesgue and Sobolev spaces we use standard notations
$\Lt(\om)$, $\H^{k}(\om)$, $k\in\n$,
and $\H(\rot,\om)$, $\H(\div,\om)$, respectively,
and introduce homogeneous boundary conditions by
$$\Hc^{k}(\om):=\overline{\Cc^{\infty}(\om)}^{\H^{k}(\om)},\qquad
\Hc(\rot,\om):=\overline{\Cc^{\infty}(\om)}^{\H(\rot,\om)},\qquad
\Hc(\div,\om):=\overline{\Cc^{\infty}(\om)}^{\H(\div,\om)}.$$

\subsubsection{Domains}
\label{sec:domains}

\begin{defi}[admissible domains]
\label{def:admdom1}
$\om$ is called 
\begin{itemize}
\item[\bf(i)]
`Friedrichs admissible' if Friedrichs' estimate
$$\exists\,\cf>0\quad
\forall\,\varphi\in\Hc^{1}(\om)\qquad
\norm{\varphi}_{\Lt(\om)}\leq\cf\norm{\na\varphi}_{\Lt(\om)}$$
\item[\bf(ii)]
`Poincar\'e admissible' if $\om$ is bounded and Poincar\'e's estimate
$$\exists\,\cp>0\quad
\forall\,\varphi\in\hH^{1}(\om)\qquad
\norm{\varphi}_{\Lt(\om)}\leq\cp\norm{\na\varphi}_{\Lt(\om)}$$
\end{itemize}
holds. Note $\hH^{1}(\om)=\H^{1}(\om)\cap\reals^{\bot}$.
\end{defi}

From now on, let $\cf$ and $\cp$ denote the best possible constants in Definition \ref{def:admdom1},
cf.~\eqref{eq:introfpest1}.

\begin{rem}[admissible domains]
\label{rem:admdom1}
We note:
\begin{itemize}
\item[\bf(i)]
Any $\om$ being bounded in at least one direction with diameter $d>0$ is Friedrichs admissible.
Moreover, we emphasise that no regularity of $\ga$ is needed.
It holds
$$\frac{1}{\cf}
=\min_{0\neq u\in\Hc^{1}(\om)}
\frac{\norm{\na u}_{\Lt(\om)}}{\norm{u}_{\Lt(\om)}},$$
which means that $1/\cf$ is the square root of the first eigenvalue of the negative Dirichlet Laplacian.
\item[\bf(ii)]
Any bounded (weak) Lipschitz domain $\om$ with diameter $d>0$ is Poincar\'e admissible.
We have
$$\frac{1}{\cp}
=\min_{0\neq u\in\hH^{1}(\om)}
\frac{\norm{\na u}_{\Lt(\om)}}{\norm{u}_{\Lt(\om)}},$$
which means that $1/\cp$ is the square root of the first positive eigenvalue of the negative Neumann Laplacian.
\end{itemize}
In both cases we have $\cf,\cp\leq d/\pi$, cf.~\cite{PW1960a}.
\end{rem}

\subsubsection{De Rham Complex}
\label{sec:appderhamshort}

In the following we shall apply the latter abstract results with different choices of $\A$
for various operators from the de Rham complex
\begin{equation}
\label{complex:derham1}
\def\arrowlength{10ex}
\def\arrowdistance{.8}
\begin{tikzcd}[column sep=\arrowlength]
\Lt(\om) 
\ar[r, rightarrow, shift left=\arrowdistance, "\A_{0}=\nac"] 
\ar[r, leftarrow, shift right=\arrowdistance, "\A_{0}^{*}=-\div"']
&
[1em]
\Lt(\om) 
\ar[r, rightarrow, shift left=\arrowdistance, "\A_{1}=\rotc"] 
\ar[r, leftarrow, shift right=\arrowdistance, "\A_{1}^{*}=\rot"']
& 
[1em]
\Lt(\om) 
\arrow[r, rightarrow, shift left=\arrowdistance, "\A_{2}=\divc"] 
\arrow[r, leftarrow, shift right=\arrowdistance, "\A_{2}^{*}=-\na"']
& 
[1em]
\Lt(\om),
\end{tikzcd}
\end{equation}
cf.~\eqref{complex:A01}, with the densely defined and closed linear operators from vector calculus
\begin{align*}
\nac:\Hc^{1}(\om)\subset\Lt(\om)&\to\Lt(\om),
&
\na:\H^{1}(\om)\subset\Lt(\om)&\to\Lt(\om);
&
\phi&\mapsto\na\phi,\\
\rotc:\Hc(\rot,\om)\subset\Lt(\om)&\to\Lt(\om),
&
\rot:\H(\rot,\om)\subset\Lt(\om)&\to\Lt(\om);
&
\Phi&\mapsto\rot\Phi,\\
\divc:\Hc(\div,\om)\subset\Lt(\om)&\to\Lt(\om),
&
\div:\H(\div,\om)\subset\Lt(\om)&\to\Lt(\om);
&
\Phi&\mapsto\div\Phi.
\end{align*}

\subsubsection{Gradients and Divergences}
\label{sec:graddiv1}

Let us start with $\na$ and $\div$.

\begin{itemize}
\item
Let $\om$ be Friedrichs admissible.
We consider 
$$\A:=\A_{0}=\nac,\qquad
\A^{*}=-\div.$$
By Friedrichs' estimate and Lemma \ref{lem:fatoolbox1}
$R(\A)$ and $R(\A^{*})$ are closed.
Moreover, 
$$N(\nac)=N(\A)=\{0\},\qquad
R(\div)=R(\A^{*})=N(\A)^{\bot}=\Lt(\om)$$
and
\begin{align*}
\cA:\Hc^{1}(\om)\subset\Lt(\om)&\to R(\A),\\
\cA^{*}:\H(\div,\om)\cap R(\A)\subset R(\A)&\to\Lt(\om).
\end{align*}
Note that by \eqref{eq:helm2}
\begin{align*}
N(\A^{*})=N(\div)&=\big\{\Phi\in\H(\div,\om):\div\Phi=0\big\},\\
R(\A)=R(\nac)=N(\rotc)\cap\Harmd(\om)^{\bot}
&=\big\{\Phi\in\Hc(\rot,\om):\rot\Phi=0\,\wedge\,\Phi\,\bot\,\Harmd(\om)\big\},
\end{align*}
where we denote the harmonic Dirichlet fields (cohomology group) by
$$\Harmd(\om)
:=N(\rotc)\cap N(\div)
=\big\{\Phi\in\Hc(\rot,\om)\cap\H(\div,\om):\rot\Phi=0\,\wedge\,\div\Phi=0\big\}.$$

$\cA$ and $\cA^{*}$ are bijective with bounded inverses
\begin{align*}
\cA^{-1}:R(\A)&\to D(\cA),
&
\norm{\cA^{-1}}_{R(\A)\to\Lt(\om)}&=\cf,\\
(\cA^{*})^{-1}:R(\A^{*})&\to D(\cA^{*}),
&
\bnorm{(\cA^{*})^{-1}}_{R(\A^{*})\to\Lt(\om)}&=\cf.
\end{align*}

Let $f\in\Lt(\om)$ and $G\in R(\nac)$.
Then 
$$u:=\cA^{-1}G\in D(\cA)=\Hc^{1}(\om),\qquad
E:=(\cA^{*})^{-1}f\in D(\cA^{*})=\H(\div,\om)\cap R(\nac)$$
are the unique solutions of the boundary value problems
$\nac u=G$ and $-\div|_{R(\nac)}E=f$, i.e.,
\begin{align*}
\na u&=G,
&
-\div E&=f
&
&\text{in }\om,\\
&
&
\rot E&=0
&
&\text{in }\om,\\
u&=0,
&
n\times E&=0
&
&\text{on }\ga,\\
&
&
E&\;\bot\;\Harmd(\om).
\end{align*}

\item
Let $\om$ be Poincar\'e admissible.
We consider 
$$\A:=-\A_{2}^{*}=\na,\qquad
\A^{*}=-\divc.$$
Note that $N(\na)=N(\A)=\reals$ and $\cA=\A|_{\reals^{\bot}}$.
By Poincar\'e's estimate and Lemma \ref{lem:fatoolbox1}
$R(\cA)=R(\A)$ and $R(\A^{*})$ are closed.
Moreover,
$$R(\divc)
=R(\A^{*})
=N(\A)^{\bot}
=\Lt(\om)\cap\reals^{\bot}
=\hLt(\om)$$
and
\begin{align*}
\cA:\hH^{1}(\om)\subset\hLt(\om)&\to R(\A),
&\hH^{1}(\om)&=\H^{1}(\om)\cap\reals^{\bot}\\
\cA^{*}:\Hc(\div,\om)\cap R(\A)\subset R(\A)&\to\hLt(\om).
\end{align*}
By \eqref{eq:helm2}
\begin{align*}
N(\A^{*})=N(\divc)&=\big\{\Phi\in\Hc(\div,\om):\div\Phi=0\big\},\\
R(\A)=R(\na)=N(\rot)\cap\Harmn(\om)^{\bot}
&=\big\{\Phi\in\H(\rot,\om):\rot\Phi=0\,\wedge\,\Phi\,\bot\,\Harmn(\om)\big\},
\end{align*}
where we denote the harmonic Neumann fields (cohomology group) by
$$\Harmn(\om)
:=N(\rot)\cap N(\divc)
=\big\{\Phi\in\H(\rot,\om)\cap\Hc(\div,\om):\rot\Phi=0\,\wedge\,\div\Phi=0\big\}.$$

$\cA$ and $\cA^{*}$ are bijective with bounded inverses
\begin{align*}
\cA^{-1}:R(\A)&\to D(\cA),
&
\norm{\cA^{-1}}_{R(\A)\to\Lt(\om)}&=\cp,\\
(\cA^{*})^{-1}:R(\A^{*})&\to D(\cA^{*}),
&
\bnorm{(\cA^{*})^{-1}}_{R(\A^{*})\to\Lt(\om)}&=\cp.
\end{align*}

Let $f\in\hLt(\om)$ and $G\in R(\na)$.
Then 
$$u:=\cA^{-1}G\in D(\cA)=\hH^{1}(\om),\qquad
E:=(\cA^{*})^{-1}f\in D(\cA^{*})=\H(\divc,\om)\cap R(\na)$$
are the unique solutions of
$\na|_{\hLt(\om)}u=G$ and $-\divc|_{R(\na)}E=f$, i.e.,
\begin{align*}
\na u&=G,
&
-\div E&=f
&
&\text{in }\om,\\
&
&
\rot E&=0
&
&\text{in }\om,\\
&
&
n\cdot E&=0
&
&\text{on }\ga,\\
u&\;\bot\;\reals,
&
E&\;\bot\;\Harmn(\om).
\end{align*}

\end{itemize}


\subsubsection{Dirichlet/Neumann Laplacians}
\label{sec:standlap}

We use the latter results.

\begin{itemize}
\item
Let $\om$ be Friedrichs admissible
and let $\A:=\A_{0}=\nac$. 
We introduce the negative Dirichlet Laplacian
\begin{align*}
\Ld:=\A^{*}\A=-\div\,\nac:
D(\Ld)\subset\Lt(\om)\to\Lt(\om);\quad
\varphi\mapsto-\Delta\varphi,
\end{align*}
where
$$D(\Ld)=\big\{\varphi\in\Hc^{1}(\om):\na\varphi\in\H(\div,\om)\big\}
=\big\{\varphi\in\Hc^{1}(\om):\Delta\varphi\in\Lt(\om)\big\}.$$
Note that 
$\cLd=\cA^{*}\cA=\A^{*}\cA=\A^{*}\A=\Ld$
with closed range $R(\Ld)=R(\A^{*})=\Lt(\om)$.
$\Ld$ is selfadjoint, positive, and bijective with bounded inverse
$$\cLd^{-1}=\cA^{-1}(\cA^{*})^{-1}:\Lt(\om)\to D(\Ld),\qquad
\norm{\cLd^{-1}}_{\Lt(\om)\to\Lt(\om)}=\cf^{2}.$$
\eqref{eq:estAA} reads
\begin{align}
\label{eq:est1lap}
\forall\,\varphi\in D(\Ld)\qquad
\norm{\varphi}_{\Lt(\om)}
\leq\cf\norm{\na\varphi}_{\Lt(\om)}
\leq\cf^{2}\norm{\Delta\varphi}_{\Lt(\om)}.
\end{align}

Let $f\in\Lt(\om)$.
Then $u:=\cLd^{-1}f\in D(\Ld)$ is the unique solution of
the Dirichlet Laplace boundary value problem $\Ld u=f$, i.e.,
\begin{align*}
-\Delta u&=f
&
&\text{in }\om,\\
u&=0
&
&\text{on }\ga.
\end{align*}
Note that 
$$u=\cA^{-1}(\cA^{*})^{-1}f=-\nac^{-1}\div|_{R(\nac)}^{-1}f,$$
cf.~Section \ref{sec:graddiv1}.
To find $u\in\Hc^{1}(\om)$ by variational methods we may consider \eqref{eq:varform2}, i.e.,
$$\forall\,\varphi\in\Hc^{1}(\om)\qquad
\scp{\na u}{\na\varphi}_{\Lt(\om)}
=\scp{f}{\varphi}_{\Lt(\om)}.$$

\item
Let $\om$ be Poincar\'e admissible
and let $\A:=-\A_{2}^{*}=\na$.
We introduce the negative Neumann Laplacian
and its reduced version
\begin{align*}
\Ln:=\A^{*}\A=-\divc\,\na:
D(\Ln)\subset\Lt(\om)&\to\Lt(\om);
&
\varphi&\mapsto-\Delta\varphi,\\
\cLn=\cA^{*}\cA=\A^{*}\cA=-\divc\,\na|_{\hLt(\om)}:
D(\cLn)\subset\hLt(\om)&\to\hLt(\om),
\end{align*}
where
\begin{align*}
D(\Ln)&=\big\{\varphi\in\H^{1}(\om):\na\varphi\in\Hc(\div,\om)\big\},\\
D(\cLn)&=\big\{\varphi\in\hH^{1}(\om):\na\varphi\in\Hc(\div,\om)\big\}.
\end{align*}
$\cLn$ is selfadjoint, positive, and bijective
with closed range $R(\Ln)=R(\A^{*})=\hLt(\om)$
and bounded inverse
$$\cLn^{-1}=\cA^{-1}(\cA^{*})^{-1}:\hLt(\om)\to D(\cLn),\qquad
\norm{\cLn^{-1}}_{\hLt(\om)\to\Lt(\om)}=\cp^{2}.$$
\eqref{eq:estAA} reads
\begin{align*}
\forall\,\varphi\in D(\cLn)\qquad
\norm{\varphi}_{\Lt(\om)}
\leq\cp\norm{\na\varphi}_{\Lt(\om)}
\leq\cp^{2}\norm{\Delta\varphi}_{\Lt(\om)}.
\end{align*}

Let $f\in\hLt(\om)$.
Then $u:=\cLn^{-1}f\in D(\Ln)$ is the unique solution of
the Neumann Laplace boundary value problem $\cLn u=f$, i.e.,
\begin{align*}
-\Delta u&=f
&
&\text{in }\om,\\
n\cdot\na u&=0
&
&\text{on }\ga,\\
u&\;\bot\;\reals.
\end{align*}
Note that 
$$u=\cA^{-1}(\cA^{*})^{-1}f=-\na|_{\hLt(\om)}^{-1}\divc|_{R(\na)}^{-1}f,$$
cf.~Section \ref{sec:graddiv1}.
To find $u\in\hH^{1}(\om)$ by variational methods we may consider \eqref{eq:varform2}, i.e.,
$$\forall\,\varphi\in\H^{1}(\om)\qquad
\scp{\na u}{\na\varphi}_{\Lt(\om)}
=\scp{f}{\varphi}_{\Lt(\om)}.$$

\end{itemize}

\subsubsection{Over- and Underdetermined Laplacians}
\label{sec:odetlap}

Let $\om$ be Friedrichs admissible.

\begin{lem}
\label{lem:friedrichsH2}
On $\Hc^{2}(\om)$ the norms 
$\norm{\cdot}_{\H^{2}(\om)}$ and $\norm{\Delta\cdot}_{\Lt(\om)}$
are equivalent. More precisely, 
$$\forall\,\varphi\in\Hc^{2}(\om)\qquad
\norm{\varphi}_{\H^{2}(\om)}
\leq\cd\norm{\Delta\varphi}_{\Lt(\om)},$$
where $\cd:=\sqrt{1+\cf^{2}+\cf^{4}}$.
\end{lem}

\begin{proof}
By \eqref{eq:est1lap}
$$\forall\,\varphi\in D(\Ld)\qquad
\norm{\varphi}_{\Lt(\om)}^{2}
+\norm{\na\varphi}_{\Lt(\om)}^{2}
+\norm{\Delta\varphi}_{\Lt(\om)}^{2}
\leq\cd^{2}\norm{\Delta\varphi}_{\Lt(\om)}^{2}.$$
For $\varphi\in\Cc^{\infty}(\om)$ we observe
$$\sum_{i,j}\scp{\p_{i}\p_{j}\varphi}{\p_{i}\p_{j}\varphi}_{\Lt(\om)}
=\sum_{i,j}\scp{\p_{i}^{2}\varphi}{\p_{j}^{2}\varphi}_{\Lt(\om)}
=\norm{\Delta\varphi}_{\Lt(\om)}^{2},$$
which extends to $\varphi\in\Hc^{2}(\om)$ by continuity
and shows the stated estimate.
\end{proof}


Lemma \ref{lem:friedrichsH2} and Lemma \ref{lem:fatoolbox1} yield that 
the over- and underdetermined Laplacians
\begin{align*}
\Lc:\Hc^{2}(\om)\subset\Lt(\om)&\to\Lt(\om);
&
\varphi&\mapsto-\Delta\varphi,\\
\L:=\Lc^{*}:\H(\Delta,\om)\subset\Lt(\om)&\to\Lt(\om)
\end{align*}
are densely defined and closed with closed ranges $R(\Lc)$ and $R(\L)$.
Moreover,
\begin{align}
\begin{aligned}
\label{eq:Rbilap}
N(\Lc)&=\{0\},
&
R(\Lc)&=N(\L)^{\bot}
=\Lt(\om)\cap\bbH^{\bot}
=\tLt(\om),\\
N(\L)&=\bbH,
&
R(\L)&=N(\Lc)^{\bot}
=\Lt(\om).
\end{aligned}
\end{align}
The reduced operators
\begin{align*}
\cLc:\Hc^{2}(\om)\subset\Lt(\om)&\to\tLt(\om),\\
\cL:\tH(\Delta,\om)\subset\tLt(\om)&\to\Lt(\om),
&\tH(\Delta,\om)&=\H(\Delta,\om)\cap\bbH^{\bot}
\end{align*}
are bijective with bounded inverse operators 
\begin{align*}
\cLc{}^{-1}:\tLt(\om)&\to D(\Lc)=\Hc^{2}(\om),\\
\cL^{-1}:\Lt(\om)&\to D(\cL)=\tH(\Delta,\om),
&
\norm{\cLc{}^{-1}}_{\tLt(\om)\to\Lt(\om)}
=\norm{\cL^{-1}}_{\Lt(\om)\to\Lt(\om)}
&\leq\cf^{2},
\end{align*}
cf.~\eqref{eq:est1lap}.

\begin{rem}
\label{rem:fpest}
In particular, we have for all 
$\varphi\in D(\cLc)=\Hc^{2}(\om)$
and all $\varphi\in D(\cL)=\tH(\Delta,\om)$
$$\norm{\varphi}_{\Lt(\om)}
\leq\cf^{2}\norm{\Delta\varphi}_{\Lt(\om)}.$$
\end{rem}

Let $f\in\tLt(\om)$.
Then $u:=\cLc{}^{-1}f\in D(\cLc)=\Hc^{2}(\om)$ 
is the unique solution of
the overdetermined negative Laplace boundary value problem $\cLc u=f$, i.e.,
\begin{align*}
-\Delta u&=f
&
&\text{in }\om,\\
u&=0
&
&\text{on }\ga,\\
n\cdot\na u&=0
&
&\text{on }\ga.
\end{align*}

\begin{rem}
\label{rem:H2bc}
Note that $u|_{\ga}=0$ implies $n\times\na u|_{\ga}=0$.
Hence, together with $n\cdot\na u|_{\ga}=0$ we see $\na u|_{\ga}=0$.
In other words, for $u\in\Hc^{1}(\om)$ it holds
$$u\in\Hc^{2}(\om)\qequi\na u\in\Hc(\div,\om)\cap\H^{1}(\om).$$ 
\end{rem}

Let $f\in\Lt(\om)$.
Then $u:=\cL^{-1}f\in D(\cL)=\tH(\Delta,\om)$ 
is the unique solution of
the underdetermined negative Laplace boundary value problem $\cL u=f$, i.e.,
\begin{align*}
-\Delta u&=f
&
&\text{in }\om,\\
u&\;\bot\;\bbH.
\end{align*}

\begin{rem}
\label{rem:minmaxlap}
Note that for the different negative Laplacians we have
$\Lc\subset\Ld,\Ln\subset\L$, that is
$$\Lc\subset\Ld=-\div\nac\subset\L,\qquad
\Lc\subset\Ln=-\divc\na\subset\L.$$
In this sense, $\Lc$ and $\L$ are minimal and maximal 
$\Lt(\om)$-realisations of the negative Laplacian, respectively.
\end{rem}

\begin{theo}
\label{theo:cptharm}
Let $\om$ be bounded. Then
$D(\cL)=\tH(\Delta,\om)\hookrightarrow\Lt(\om)$
is compact.
\end{theo}

\begin{proof}
By Lemma \ref{lem:fatoolbox1} we have that the embedding 
$D(\cLc)=D(\Lc)=\Hc^{2}(\om)\hookrightarrow\Lt(\om)$
is compact if and only if the embedding 
$D(\cL)=\tH(\Delta,\om)\hookrightarrow\Lt(\om)$
is compact.
Hence the latter embedding is compact by Rellich's selection theorem
for, e.g., $\Hc^{1}(\om)$.
\end{proof}

\subsubsection{Bi-Laplacians and Biharmonic Operators}
\label{sec:biharm1}

Let $\om$ be Friedrichs admissible
and let us consider $\A=\Lc$ 
with $\A^{*}=\L$ from the latter section.

\begin{itemize}

\item
We introduce the Dirichlet bi-Laplacian (Dirichlet biharmonic operator)
$$\Bd:=\A^{*}\A=\L\Lc:
D(\Bd)\to\Lt(\om);\quad
\varphi\mapsto\Delta^{2}\varphi,$$
where
$$D(\Bd)=\big\{u\in\Hc^{2}(\om):\Delta u\in\H(\Delta,\om)\big\}
=\big\{u\in\Hc^{2}(\om):\Delta^{2}u\in\Lt(\om)\big\}.$$
Then
$\cBd=\cA^{*}\cA=\cL\cLc=\L\cLc=\L\Lc=\Bd$
with closed range $R(\Bd)=R(\L)=\Lt(\om)$.
$\Bd$ is selfadjoint, positive, and bijective with bounded inverse
$$\cBd^{-1}:=\cA^{-1}(\cA^{*})^{-1}=\cLc{}^{-1}\cL^{-1}:
\Lt(\om)\to D(\Bd),\qquad
\norm{\Bd^{-1}}_{\Lt(\om)\to\Lt(\om)}
\leq\cf^{4}.$$
\eqref{eq:estAA} reads
\begin{align*}
\forall\,\varphi\in D(\Bd)\qquad
\norm{\varphi}_{\Lt(\om)}
\leq\cf^{2}\norm{\Delta\varphi}_{\Lt(\om)}
\leq\cf^{4}\norm{\Delta^{2}\varphi}_{\Lt(\om)}.
\end{align*}

Let $f\in\Lt(\om)$.
Then $u:=\cBd^{-1}f\in D(\Bd)$ is the unique solution of
the Dirichlet boundary value problem for the bi-Laplacian $\Bd u=f$, i.e.,
\begin{align*}
\Delta^{2}u&=f
&
&\text{in }\om,\\
u&=0
&
&\text{on }\ga,\\
n\cdot\na u&=0
&
&\text{on }\ga.
\end{align*}
Note that 
$u=\cLc{}^{-1}\cL^{-1}f$.
To find $u\in\Hc^{2}(\om)$ by variational methods we may consider \eqref{eq:varform2}, i.e.,
\begin{align}
\label{eq:varBD}
\forall\,\varphi\in\Hc^{2}(\om)\qquad
\scp{\Delta u}{\Delta\varphi}_{\Lt(\om)}
=\scp{ f}{\varphi}_{\Lt(\om)}.
\end{align}

\item
Analogously, we may consider the Neumann bi-Laplacian (Neumann biharmonic operator)
and its reduced version
\begin{align*}
\Bn:=\A\A^{*}=\Lc\L:
D(\Bn)&\to\Lt(\om),\\
\cBn:=\cA\cA^{*}=\cLc\cL:
D(\cBn)&\to R(\Lc)=\tLt(\om),
\end{align*}
where we recall \eqref{eq:Rbilap} and 
\begin{align*}
D(\Bn)&=\big\{u\in\H(\Delta,\om):\Delta u\in\Hc^{2}(\om)\big\},\\
D(\cBn)&=\big\{u\in\tH(\Delta,\om):\Delta u\in\Hc^{2}(\om)\big\}.
\end{align*}
$\Bn$ and $\cBn$ have closed range $R(\Bn)=R(\Lc)=\tLt(\om)$.
$\cBn$ is selfadjoint, positive, and bijective with bounded inverse
$$\cBn^{-1}:=(\cA^{*})^{-1}\cA^{-1}=\cL^{-1}\cLc{}^{-1}:
\tLt(\om)\to D(\cBn),\qquad
\norm{\cBn^{-1}}_{\tLt(\om)\to\Lt(\om)}=\cf^{4}.$$
\eqref{eq:estAA} reads
\begin{align*}
\forall\,\varphi\in D(\cBn)\qquad
\norm{\varphi}_{\Lt(\om)}
\leq\cf^{2}\norm{\Delta\varphi}_{\Lt(\om)}
\leq\cf^{4}\norm{\Delta^{2}\varphi}_{\Lt(\om)}.
\end{align*}

Let $f\in\tLt(\om)$.
Then $u:=\cBn^{-1}f\in D(\cBn)$ is the unique solution of
the Neumann boundary value problem for the bi-Laplacian $\cBn u=f$, i.e.,
\begin{align*}
\Delta^{2}u&=f
&
&\text{in }\om,\\
\Delta u&=0
&
&\text{on }\ga,\\
n\cdot\na\Delta u&=0
&
&\text{on }\ga,\\
u&\;\bot\;\bbH.
\end{align*}
As in Remark \ref{rem:H2bc} we have $\na\Delta u|_{\ga}=0$.
Note that 
$u=\cL^{-1}\cLc{}^{-1}f$.
To find $u\in\tH(\Delta,\om)$ by variational methods we may consider \eqref{eq:varform2}, i.e.,
\begin{align}
\label{eq:varBN}
\forall\,\varphi\in\H(\Delta,\om)\qquad
\scp{\Delta u}{\Delta\varphi}_{\Lt(\om)}
=\scp{ f}{\varphi}_{\Lt(\om)}.
\end{align}

\end{itemize}

\begin{rem}[over- and underdetermined Laplacians]
\label{rem:biharmoverdetlap}
Recall Section \ref{sec:odetlap}.
\begin{itemize}
\item
For $u:=\Bd^{-1}f=\cLc{}^{-1}\cL^{-1}f$
we get that $\wt{u}:=\Delta u=\cL^{-1}f$
is the unique solution of the underdetermined Laplace problem.
Hence we may solve the underdetermined Laplace problem
using a variational formulation for the Dirichlet bi-Laplace problem
\eqref{eq:varBD}.
In fact, we compute a potential $u$ and set $\wt{u}:=\cLc u=\Delta u$,
cf.~Lemma \ref{lem:fatoolbox2}(iii).
\item
For $u:=\cBn^{-1}f=\cL^{-1}\cLc{}^{-1}f$
we get that $\wt{u}:=\Delta u=\cLc{}^{-1}f$
is the unique solution of the overdetermined Laplace problem.
Thus we can solve the overdetermined Laplace problem
using a variational formulation for the Neumann bi-Laplace problem
\eqref{eq:varBN}.
Here, we compute a potential $u$ and set $\wt{u}:=\cL u=\Delta u$,
cf.~Lemma \ref{lem:fatoolbox2}(iii).
\end{itemize}
\end{rem}

\subsubsection{A Zoo of Biharmonic Operators}
\label{sec:zoo1}

Let $\om$ be Poincar\'e admissible.
Note that $\reals\subset\bbH$ and hence 
$$\tLt(\om)\subset\hLt(\om).$$
Let us recall the different reduced negative Laplacians
$$\cLc,\qquad
\cLd=\Ld=-\div\nac,\qquad
\cLn=-\divc\na_{\hLt(\om)},\qquad
\cL$$
with domains of definition
\begin{align*}
D(\cLc)&=\Hc^{2}(\om),
&
D(\cLd)&=\big\{\varphi\in\Hc^{1}(\om):\Delta\varphi\in\Lt(\om)\big\},\\
D(\cL)&=\tH(\Delta,\om),
&
D(\cLn)&=\big\{\varphi\in\hH^{1}(\om):\na\varphi\in\Hc(\div,\om)\big\}
\end{align*}
and their respective bounded inverse operators 
\begin{align}
\begin{aligned}
\label{eq:cLinv}
\cLc{}^{-1}:\tLt(\om)&\to D(\cLc)\subset\Lt(\om),
&
\cLd^{-1}:\Lt(\om)&\to D(\cLd)\subset\Lt(\om),\\
\cL^{-1}:\Lt(\om)&\to D(\cL)\subset\tLt(\om),
&
\cLn^{-1}:\hLt(\om)&\to D(\cLn)\subset\hLt(\om).
\end{aligned}
\end{align}

In Section \ref{sec:biharm1} we already discussed 
the Dirichlet and Neumann biharmonic operators
$$\cB_{\cdot,\circ}:=\cBd=\cL\cLc=\L\Lc=\Bd,\qquad
\cB_{\circ,\cdot}:=\cBn=\cLc\cL,$$
being selfadjoint, positive, bijective and boundedly invertible
with the well posed (well defined and uniquely solvable) 
inverse operators $\cLc{}^{-1}\cL^{-1}$ and $\cL^{-1}\cLc{}^{-1}$.

Combining the four different Laplacians we obtain 
a whole zoo of formally sixteen biharmonic operators.
Due to the restrictions of $\hLt(\om)$ and $\tLt(\om)$
some combinations are -- even formally -- not possible
(without further restrictions), those are the five combinations
$$\cLc{}^{-1}\cLc{}^{-1},\qquad
\cLc{}^{-1}\cLd^{-1},\qquad
\cLc{}^{-1}\cLn^{-1},\qquad
\cLn^{-1}\cLc{}^{-1},\qquad
\cLn^{-1}\cLd^{-1},$$
corresponding to 
$\cLc\cLc$, $\cLd\cLc$, $\cLn\cLc$, $\cLc\cLn$, $\cLd\cLn$, respectively.
One has to consider different operators to realise even stronger or weaker boundary conditions.
We come back to this later.

We end up with only nine more well posed 
biharmonic operators, namely
\begin{align*}
\cB_{\cdot,\cdot}&:=\cL\cL,
&
\cB_{\ttD,\cdot}&:=\cLd\cL,
&
\cB_{\ttN,\cdot}&:=\cLn\cL,
&
\cB_{\cdot,\ttN}&:=\cL\cLn,
&
\cBnn&:=\cLn\cLn,\\
\cB_{\circ,\ttD}&:=\cLc\cLd,
&
\cB_{\cdot,\ttD}&:=\cL\cLd,
&
\cBdd&:=\cLd\cLd,
&
\cB_{\ttN,\ttD}&:=\cLn\cLd.
\end{align*}
All eleven biharmonic operators are
bijective and boundedly invertible, more precisely
\begin{align*}
\cB_{\cdot,\circ}^{-1}
=\cLc{}^{-1}\cL^{-1}
:\Lt(\om)&\to\Lt(\om),
&
\cB_{\circ,\cdot}^{-1}
=\cL^{-1}\cLc{}^{-1}
:\tLt(\om)&\to\tLt(\om),\\
\cB_{\cdot,\cdot}^{-1}
=\cL^{-1}\cL^{-1}
:\Lt(\om)&\to\tLt(\om),
&
\cB_{\ttD,\cdot}^{-1}
=\cL^{-1}\cLd^{-1}
:\Lt(\om)&\to\tLt(\om),\\
\cB_{\ttN,\cdot}^{-1}
=\cL^{-1}\cLn^{-1}
:\hLt(\om)&\to\tLt(\om),
&
\cB_{\cdot,\ttN}^{-1}
=\cLn^{-1}\cL^{-1}
:\Lt(\om)&\to\hLt(\om),\\
\cBnn^{-1}
=\cLn^{-1}\cLn^{-1}
:\hLt(\om)&\to\hLt(\om),
&
\cB_{\circ,\ttD}^{-1}
=\cLd^{-1}\cLc{}^{-1}
:\tLt(\om)&\to\Lt(\om),\\
\cB_{\cdot,\ttD}^{-1}
=\cLd^{-1}\cL^{-1}
:\Lt(\om)&\to\Lt(\om),
&
\cBdd^{-1}
=\cLd^{-1}\cLd^{-1}
:\Lt(\om)&\to\Lt(\om),\\
\cB_{\ttN,\ttD}^{-1}
=\cLd^{-1}\cLn^{-1}
:\hLt(\om)&\to\Lt(\om).
\end{align*}

Let us write down the classical formulations of the latter eleven 
(some apparently over- and underdetermined) biharmonic operators
as uniquely solvable boundary value problems for the bi-Laplacian:

\begin{enumerate}[\sf\bfseries(i)]

\item
$\cB_{\cdot,\circ}^{-1}
=\cLc{}^{-1}\cL^{-1}:
\Lt(\om)\to\big\{u\in D(\cLc)=\Hc^{2}(\om):
\Delta u\in D(\cL)=\tH(\Delta,\om)\big\}$\\
yields the unique solution $u$ of
\begin{align*}
\Delta^{2}u&=f\in\Lt(\om)
&
&\text{in }\om,\\
u&=0
&
&\text{on }\ga,\\
n\cdot\na u&=0
&
&\text{on }\ga,\\
\Delta u&\;\bot\;\bbH.
\end{align*}
This is the Dirichlet biharmonic problem from
\eqref{eq:bcdef1}, \eqref{eq:bcdef2}.
The last condition is redundant.

\item
$\cB_{\circ,\cdot}^{-1}
=\cL^{-1}\cLc{}^{-1}:
\tLt(\om)\to\big\{u\in D(\cL)=\tH(\Delta,\om):
\Delta u\in D(\cLc)=\Hc^{2}(\om)\big\}$\\
yields the unique solution $u$ of
\begin{align*}
\Delta^{2}u&=f\in\tLt(\om)
&
&\text{in }\om,\\
\Delta u&=0
&
&\text{on }\ga,\\
n\cdot\na\Delta u&=0
&
&\text{on }\ga,\\
u&\;\bot\;\bbH.
\end{align*}
This is the Neumann biharmonic problem from
\eqref{eq:bcdef1}, \eqref{eq:bcdef2}.

\item
$\cB_{\cdot,\cdot}^{-1}
=\cL^{-1}\cL^{-1}:
\Lt(\om)\to 
\big\{u\in D(\cL)=\tH(\Delta,\om):
\Delta u\in D(\cL)=\tH(\Delta,\om)\big\}$\\
yields the unique solution $u$ of
\begin{align*}
\Delta^{2}u&=f\in\Lt(\om)
&
&\text{in }\om,\\
u&\;\bot\;\bbH,\\
\Delta u&\;\bot\;\bbH.
\end{align*}

\item
$\cB_{\ttD,\cdot}^{-1}
=\cL^{-1}\cLd^{-1}:
\Lt(\om)\to 
\big\{u\in D(\cL)=\tH(\Delta,\om):
\Delta u\in D(\cLd)\big\}$\\
yields the unique solution $u$ of
\begin{align*}
\Delta^{2}u&=f\in\Lt(\om)
&
&\text{in }\om,\\
\Delta u&=0
&
&\text{on }\ga,\\
u&\;\bot\;\bbH.
\end{align*}

\item
$\cB_{\ttN,\cdot}^{-1}
=\cL^{-1}\cLn^{-1}:
\hLt(\om)\to 
\big\{u\in D(\cL)=\tH(\Delta,\om):
\Delta u\in D(\cLn)\big\}$\\
yields the unique solution $u$ of
\begin{align*}
\Delta^{2}u&=f\in\hLt(\om)
&
&\text{in }\om,\\
n\cdot\na\Delta u&=0
&
&\text{on }\ga,\\
u&\;\bot\;\bbH,\\
\Delta u&\;\bot\;\reals.
\end{align*}

\item
$\cB_{\cdot,\ttN}^{-1}
=\cLn^{-1}\cL^{-1}:
\Lt(\om)\to 
\big\{u\in D(\cLn):
\Delta u\in D(\cL)=\tH(\Delta,\om)\big\}$\\
yields the unique solution $u$ of
\begin{align*}
\Delta^{2}u&=f\in\Lt(\om)
&
&\text{in }\om,\\
n\cdot\na u&=0
&
&\text{on }\ga,\\
u&\;\bot\;\reals,\\
\Delta u&\;\bot\;\bbH.
\end{align*}

\item
$\cBnn^{-1}
=\cLn^{-1}\cLn^{-1}:
\hLt(\om)\to 
\big\{u\in D(\cLn):
\Delta u\in D(\cLn)\big\}$\\
yields the unique solution $u$ of
\begin{align*}
\Delta^{2}u&=f\in\hLt(\om)
&
&\text{in }\om,\\
n\cdot\na u&=0
&
&\text{on }\ga,\\
n\cdot\na\Delta u&=0
&
&\text{on }\ga,\\
u&\;\bot\;\reals,\\
\Delta u&\;\bot\;\reals.
\end{align*}
This is the Riquier biharmonic problem from
\eqref{eq:bcdef1}, \eqref{eq:bcdef2}.
The last condition is redundant.

\item
$\cB_{\circ,\ttD}^{-1}
=\cLd^{-1}\cLc{}^{-1}:
\tLt(\om)\to\big\{u\in D(\cLd):
\Delta u\in D(\cLc)=\Hc^{2}(\om)\big\}$\\
yields the unique solution $u$ of
\begin{align*}
\Delta^{2}u&=f\in\tLt(\om)
&
&\text{in }\om,\\
u&=0
&
&\text{on }\ga,\\
\Delta u&=0
&
&\text{on }\ga,\\
n\cdot\na\Delta u&=0
&
&\text{on }\ga.
\end{align*}

\item
$\cB_{\cdot,\ttD}^{-1}
=\cLd^{-1}\cL^{-1}:
\Lt(\om)\to\big\{u\in D(\cLd):
\Delta u\in D(\cL)=\tH(\Delta,\om)\big\}$\\
yields the unique solution $u$ of
\begin{align*}
\Delta^{2}u&=f\in\Lt(\om)
&
&\text{in }\om,\\
u&=0
&
&\text{on }\ga,\\
\Delta u&\;\bot\;\bbH.
\end{align*}

\item
$\cBdd^{-1}
=\cLd^{-1}\cLd^{-1}:
\Lt(\om)\to 
\big\{u\in D(\cLd):
\Delta u\in D(\cLd)\big\}$\\
yields the unique solution $u$ of
\begin{align*}
\Delta^{2}u&=f\in\Lt(\om)
&
&\text{in }\om,\\
u&=0
&
&\text{on }\ga,\\
\Delta u&=0
&
&\text{on }\ga.
\end{align*}
This is the Navier biharmonic problem from
\eqref{eq:bcdef1}, \eqref{eq:bcdef2}.

\item
$\cB_{\ttN,\ttD}^{-1}
=\cLd^{-1}\cLn^{-1}:
\hLt(\om)\to 
\big\{u\in D(\cLd):
\Delta u\in D(\cLn)\big\}$\\
yields the unique solution $u$ of
\begin{align*}
\Delta^{2}u&=f\in\hLt(\om)
&
&\text{in }\om,\\
u&=0
&
&\text{on }\ga,\\
n\cdot\na\Delta u&=0
&
&\text{on }\ga,\\
\Delta u&\;\bot\;\reals.
\end{align*}

\end{enumerate}

\subsubsection{Over- and Underdetermined Biharmonic Operators}
\label{sec:odetbiharm}

Let $\om$ be Friedrichs admissible.
We shall follow the rationale from Section \ref{sec:odetlap}.
Note that \eqref{eq:est1lap} shows
\begin{align}
\label{eq:est1bilap}
\forall\,\varphi\in\Hc^{4}(\om)\qquad
\norm{\varphi}_{\Lt(\om)}
\leq\cf\norm{\na\varphi}_{\Lt(\om)}
\leq\cf^{2}\norm{\Delta\varphi}_{\Lt(\om)}
\leq\cf^{3}\norm{\Delta\na\varphi}_{\Lt(\om)}
\leq\cf^{4}\norm{\Delta^{2}\varphi}_{\Lt(\om)},
\end{align}
which holds also for all $\varphi\in\Hc^{3}(\om)$
with $\Delta^{2}\varphi\in\Lt(\om)$.

\begin{lem}
\label{lem:friedrichsH4}
On $\Hc^{4}(\om)$ the norms 
$\norm{\cdot}_{\H^{4}(\om)}$ and $\norm{\Delta^{2}\cdot}_{\Lt(\om)}$
are equivalent. More precisely, 
$$\forall\,\varphi\in\Hc^{4}(\om)\qquad
\norm{\varphi}_{\H^{4}(\om)}
\leq\cdd\norm{\Delta^{2}\varphi}_{\Lt(\om)},$$
where $\cdd:=\cd\sqrt{1+\cf^{4}}=\sqrt{1+\cf^{2}+2\cf^{4}+\cf^{6}+\cf^{8}}$.
\end{lem}

\begin{proof}
Note that $\varphi\in\Hc^{4}(\om)$ implies
$\varphi,\,\p_{i}\p_{j}\varphi\in\Hc^{2}(\om)$.
Lemma \ref{lem:friedrichsH2}, its proof, and \eqref{eq:est1lap} show
\begin{align*}
\norm{\varphi}_{\H^{4}(\om)}^{2}
=\norm{\varphi}_{\H^{2}(\om)}^{2}
+\sum_{i,j}\norm{\p_{i}\p_{j}\varphi}_{\H^{2}(\om)}^{2}
\leq\cd^{2}\Big(\hspace*{-2ex}\underbrace{\norm{\Delta\varphi}_{\Lt(\om)}^{2}}_{\displaystyle\leq\cf^{4}\norm{\Delta^{2}\varphi}_{\Lt(\om)}^{2}}
+\hspace*{1ex}\underbrace{\sum_{i,j}\norm{\p_{i}\p_{j}\Delta\varphi}_{\Lt(\om)}^{2}}_{\displaystyle=\norm{\Delta^{2}\varphi}_{\Lt(\om)}^{2}}\Big),
\end{align*}
completing the proof.
\end{proof}


Lemma \ref{lem:friedrichsH4} and Lemma \ref{lem:fatoolbox1} show that 
the over- and underdetermined biharmonic operators (bi-Laplacians)
\begin{align*}
\Bc:\Hc^{4}(\om)\subset\Lt(\om)&\to\Lt(\om);
&
\varphi&\mapsto\Delta^{2}\varphi,\\
\B:=\Bc^{*}:\H(\Delta^{2},\om)\subset\Lt(\om)&\to\Lt(\om)
\end{align*}
are densely defined and closed with closed ranges $R(\Bc)$ and $R(\B)$.
Moreover,
\begin{align}
\begin{aligned}
\label{eq:Rbibiharm}
N(\Bc)&=\{0\},
&
R(\Bc)&=N(\B)^{\bot}
=\Lt(\om)\cap\bbBH^{\bot}
=\cLt(\om),\\
N(\B)&=\bbBH,
&
R(\B)&=N(\Bc)^{\bot}
=\Lt(\om).
\end{aligned}
\end{align}
The reduced operators
\begin{align*}
\cBc:\Hc^{4}(\om)\subset\Lt(\om)&\to\cLt(\om),\\
\cB:\cH(\Delta^{2},\om)\subset\cLt(\om)&\to\Lt(\om),
&
\cH(\Delta^{2},\om)&=\H(\Delta^{2},\om)\cap\bbBH^{\bot}
\end{align*}
are bijective with bounded inverse operators 
\begin{align*}
\cBc{}^{-1}:\cLt(\om)&\to D(\Bc)=\Hc^{4}(\om),\\
\cB^{-1}:\Lt(\om)&\to D(\cB)=\cH(\Delta^{2},\om),
&
\norm{\cBc{}^{-1}}_{\cLt(\om)\to\Lt(\om)}
=\norm{\cB^{-1}}_{\Lt(\om)\to\Lt(\om)}
&\leq\cf^{4},
\end{align*}
cf.~\eqref{eq:est1bilap}.

Let $f\in\cLt(\om)$.
Then $u:=\cBc{}^{-1}f\in D(\Bc)=\Hc^{4}(\om)$ 
is the unique solution of
the overdetermined biharmonic boundary value problem $\cBc u=f$, i.e.,
\begin{align*}
\Delta^{2}u&=f
&
&\text{in }\om,\\
\forall\,|\alpha|\leq3\qquad\p^{\alpha}v&=0
&
&\text{on }\ga.
\end{align*}

Let $f\in\Lt(\om)$.
Then $u:=\cB^{-1}f\in D(\cB)=\cH(\Delta,\om)$ 
is the unique solution of
the underdetermined biharmonic boundary value problem $\cB u=f$, i.e.,
\begin{align*}
\Delta^{2}u&=f
&
&\text{in }\om,\\
u&\;\bot\;\bbBH.
\end{align*}

Note that the latter two boundary value problems for $\Delta^{2}$
complete the list from Section \ref{sec:zoo1}.
Let us call them number {\sf\bfseries(xii)} and {\sf\bfseries(xiii)}.

\begin{rem}
\label{rem:minmaxbiharm}
Note that for the different biharmonic operators we have
$\Bc\subset\Bd,\Bn\subset\B$, that is
$$\Bc\subset\Bd=\L\Lc\subset\B,\qquad
\Bc\subset\Bn=\Lc\L\subset\B.$$
In this sense, $\Bc$ and $\B$ are minimal and maximal 
$\Lt(\om)$-realisations of the biharmonic operator, respectively.
\end{rem}

Recall Theorem \ref{theo:cptharm}.

\begin{theo}
\label{theo:cptbiharm}
Let $\om$ be bounded. Then
$D(\cB)=\cH(\Delta^{2},\om)\hookrightarrow\Lt(\om)$
is compact.
\end{theo}

\begin{proof}
By Lemma \ref{lem:fatoolbox1} we have that the embedding 
$D(\cBc)=D(\Bc)=\Hc^{4}(\om)\hookrightarrow\Lt(\om)$
is compact if and only if the embedding 
$D(\cB)=\cH(\Delta^{2},\om)\hookrightarrow\Lt(\om)$
is compact.
Hence the latter embedding is compact by Rellich's selection theorem
for, e.g., $\Hc^{1}(\om)$.
\end{proof}

\subsubsection{Biharmonic Operators with Mixed Boundary Conditions}
\label{sec:mixedbc1}

In principle, everything works also with mixed boundary conditions.
For example, let us consider a division of $\ga$ into
two relatively open parts $\gat\neq\emptyset$ and $\gan:=\ga\setminus\overline{\gat}$.
We introduce $\H^{1}_{\gat}(\om)$ as closure of $\Cc_{\gat}^{\infty}(\rt)$
(the compact support does not touch $\gat$) in $\H^{1}(\om)$.
Analogously we define $\H_{\gan}(\div,\om)$.

Let $\om$ be such that the embedding 
\begin{align}
\label{eq:cptH1}
\H^{1}(\om)\hookrightarrow\Lt(\om)
\end{align}
is compact, which holds, e.g., for bounded Lipschitz domains $\om$.

By the compact embedding \eqref{eq:cptH1} we obtain
the Friedrichs/Poincar\'e estimate
\begin{align}
\label{eq:fpest1}
\exists\,\cfp>0\quad
\forall\,\varphi\in\H^{1}_{\gat}(\om)\qquad
\norm{\varphi}_{\Lt(\om)}\leq\cfp\norm{\na\varphi}_{\Lt(\om)},
\end{align}
cf.~Definition \ref{def:admdom1}.
As before, we assume $\cfp$ to be the best possible constant.

Then
$$\A:=\na_{\gat}:\H^{1}_{\gat}(\om)\subset\Lt(\om)\to\Lt(\om)$$
is densely defined and closed with adjoint 
$$\A^{*}=\na_{\gat}^{*}=-\div_{\gan}:\H_{\gan}(\div,\om)\subset\Lt(\om)\to\Lt(\om),$$
cf.~\cite{BPS2016a,PS2022a}.
Note that the cases $\gat=\ga$ and $\gat=\emptyset$
have already been discussed by $\nac=\naga$ and $\na=\naes$, respectively.
By \eqref{eq:fpest1} and Lemma \ref{lem:fatoolbox1} 
$R(\A)$ and $R(\A^{*})$ are closed 
and we have
$$N(\na_{\gat})=N(\A)=\{0\},\qquad
R(\div_{\gan})=R(\A^{*})=N(\A)^{\bot}=\Lt(\om).$$
The full de Rham complex reads, cf.~\eqref{complex:A01} and \eqref{complex:derham1},
\begin{equation}
\label{complex:derham2}
\def\arrowlength{10ex}
\def\arrowdistance{.8}
\begin{tikzcd}[column sep=\arrowlength]
\Lt(\om) 
\ar[r, rightarrow, shift left=\arrowdistance, "\A_{0}=\na_{\gat}"] 
\ar[r, leftarrow, shift right=\arrowdistance, "\A_{0}^{*}=-\div_{\gan}"']
&
[2em]
\Lt(\om) 
\ar[r, rightarrow, shift left=\arrowdistance, "\A_{1}=\rot_{\gat}"] 
\ar[r, leftarrow, shift right=\arrowdistance, "\A_{1}^{*}=\rot_{\gan}"']
& 
[2em]
\Lt(\om) 
\arrow[r, rightarrow, shift left=\arrowdistance, "\A_{2}=\div_{\gat}"] 
\arrow[r, leftarrow, shift right=\arrowdistance, "\A_{2}^{*}=-\na_{\gan}"']
& 
[2em]
\Lt(\om).
\end{tikzcd}
\end{equation}

The negative Dirichlet-Neumann-Laplacian
$$-\Delta_{\gat}:=\cA^{*}\cA=\A^{*}\A=-\div_{\gan}\,\na_{\gat}:
D(\Delta_{\gat})\subset\Lt(\om)\to\Lt(\om),$$
where
$$D(\Delta_{\gat})=\big\{\varphi\in\H^{1}_{\gat}(\om):\na\varphi\in\H_{\gan}(\div,\om)\big\},$$
is selfadjoint, positive, and bijective with bounded inverse
$$-\Delta_{\gat}^{-1}=\cA^{-1}(\cA^{*})^{-1}:\Lt(\om)\to D(\Delta_{\gat}),\qquad
\bnorm{\Delta_{\gat}^{-1}}_{\Lt(\om)\to\Lt(\om)}=\cfp^{2}.$$
\eqref{eq:estAA} reads
\begin{align}
\label{eq:est1lapmixed}
\forall\,\varphi\in D(\Delta_{\gat})\qquad
\norm{\varphi}_{\Lt(\om)}
\leq\cfp\norm{\na\varphi}_{\Lt(\om)}
\leq\cfp^{2}\norm{\Delta\varphi}_{\Lt(\om)}.
\end{align}

Let $f\in\Lt(\om)$.
Then $u:=-\Delta_{\gat}^{-1}f\in D(\Delta_{\gat})$ is the unique solution of
the Dirichlet-Neumann Laplace boundary value problem $-\Delta_{\gat} u=f$, i.e.,
\begin{align*}
\Delta u&=f
&
&\text{in }\om,\\
u&=0
&
&\text{on }\gat,\\
n\cdot\na u&=0
&
&\text{on }\gan.
\end{align*}
To find $u\in\H^{1}_{\gat}(\om)$ by variational methods we may consider \eqref{eq:varform2}, i.e.,
$$\forall\,\varphi\in\H^{1}_{\gat}(\om)\qquad
\scp{\na u}{\na\varphi}_{\Lt(\om)}
=\scp{f}{\varphi}_{\Lt(\om)}.$$
Moreover, $v:=\Delta_{\gat}^{-1}\Delta_{\gamma_{\mathsf{t}}}^{-1}f\in D(\Delta_{\gamma_{\mathsf{t}}}\Delta_{\gat})$
for some other boundary pair 
$\gamma_{\mathsf{t}}\neq\emptyset$ and $\gamma_{\mathsf{n}}:=\ga\setminus\overline{\gamma_{\mathsf{t}}}$
is the unique solution of
the Dirichlet-Neumann biharmonic boundary value problem $\Delta_{\gamma_{\mathsf{t}}}\Delta_{\gat}v=f$, i.e.,
\begin{align*}
\Delta v^{2}&=f
&
&\text{in }\om,\\
v&=0
&
&\text{on }\gat,\\
n\cdot\na v&=0
&
&\text{on }\gan,\\
\Delta v&=0
&
&\text{on }\gamma_{\mathsf{t}},\\
n\cdot\na\Delta v&=0
&
&\text{on }\gamma_{\mathsf{n}}.
\end{align*}

\subsection{Applications to the Hessian Complex}
\label{sec:appbiharm}

Let $\om$ be a bounded Lipschitz domain,
and recall the boundary parts $\gat$ and $\gan$ 
and the definition of the Sobolev spaces as closures of test fields from the latter section.
In the following we shall apply our theory to the
Hessian complex (complex of the biharmonic equation and general relativity)
\begin{equation}
\label{complex:biharm}
\def\arrowlength{10ex}
\def\arrowdistance{.8}
\begin{tikzcd}[column sep=\arrowlength]
\Lt(\om) 
\ar[r, rightarrow, shift left=\arrowdistance, "\A_{0}=\nana_{\gat}"] 
\ar[r, leftarrow, shift right=\arrowdistance, "\A_{0}^{*}=\divDiv_{\bbS,\gan}"']
&
[3em]
\Lt_{\bbS}(\om) 
\ar[r, rightarrow, shift left=\arrowdistance, "\A_{1}=\Rot_{\bbS,\gat}"] 
\ar[r, leftarrow, shift right=\arrowdistance, "\A_{1}^{*}=\symRot_{\bbT,\gan}"']
& 
[3em]
\Lt_{\bbT}(\om) 
\arrow[r, rightarrow, shift left=\arrowdistance, "\A_{2}=\Div_{\bbT,\gat}"] 
\arrow[r, leftarrow, shift right=\arrowdistance, "\A_{2}^{*}=-\devna_{\gan}"']
& 
[3em]
\Lt(\om),
\end{tikzcd}
\end{equation}
cf.~\eqref{complex:A01} and \eqref{complex:derham1}, 
with the densely defined and closed linear operators
(acting row-wise on tensors)
\begin{align*}
\nana_{\gat}:\H^{2}_{\gat}(\om)\subset\Lt(\om)&\to\Lt_{\bbS}(\om);
&
\phi&\mapsto\nana\phi,\\
\Rot_{\bbS,\gat}:\H_{\bbS,\gat}(\Rot,\om)\subset\Lt_{\bbS}(\om)&\to\Lt_{\bbT}(\om);
&
\Phi&\mapsto\Rot\Phi,\\
\Div_{\bbT,\gat}:\H_{\bbT,\gat}(\Div,\om)\subset\Lt_{\bbT}(\om)&\to\Lt(\om);
&
\Psi&\mapsto\Div\Psi,\\
\divDiv_{\bbS,\gan}:\H_{\bbS,\gan}(\divDiv,\om)\subset\Lt_{\bbS}(\om)&\to\Lt(\om);
&
\Phi&\mapsto\divDiv\Phi,\\
\symRot_{\bbT,\gan}:\H_{\bbT,\gan}(\symRot,\om)\subset\Lt_{\bbT}(\om)&\to\Lt_{\bbS}(\om);
&
\Psi&\mapsto\Rot\Psi,\\
\devna_{\gan}:\H^{1}_{\gan}(\om)\subset\Lt(\om)&\to\Lt_{\bbT}(\om);
&
\psi&\mapsto\nana\psi,
\end{align*}
cf.~\cite{PS2024a,PZ2020a} for well-posedness.

Recall the identity \eqref{eq:DeltadivDivnana}.
Hence, another way to look at the biharmonic equation 
-- underlining more the geometry (complex property) of the underlying operators --
is to investigate the biharmonic operator 
$$\B_{\gat}:=\A_{0}^{*}\A_{0}=\divDiv_{\bbS,\gan}\nana_{\gat}:
D(\A_{0}^{*}\A_{0})\subset\Lt(\om)\to\Lt(\om);\qquad
\Phi\mapsto\divDiv\nana\Phi,$$
where
$$D(\B_{\gat})
=\big\{\varphi\in\H^{2}_{\gat}(\om):\nana\varphi\in\H_{\bbS,\gan}(\divDiv,\om)\big\}.$$
For $\gat=\ga$ we get back the Dirichlet biharmonic operator, 
i.e., $\B_{\ga}=\Bd$.
But if $\gat\neq\ga$ we obtain a different operator due to the boundary conditions
being imposed on the scalars $u$ and $\Delta u$ for $\Delta^{2}$, and on the other hand 
on the scalar $u$ and the symmetric tensor $S:=\nana u$ 
for $\B_{\gat}$. 

For simplicity, let us assume $\emptyset\neq\gat\neq\ga$.
The Friedrichs/Poincar\'e estimate \eqref{eq:fpest1} yields
\begin{align}
\label{eq:fpest2}
\forall\,\varphi\in\H^{2}_{\gat}(\om)\qquad
\norm{\varphi}_{\Lt(\om)}
\leq\cfp\norm{\na\varphi}_{\Lt(\om)}
\leq\cfp^{2}\norm{\nana\varphi}_{\Lt(\om)}.
\end{align}
By \eqref{eq:fpest2} and Lemma \ref{lem:fatoolbox1}
$R(\A_{0})$ and $R(\B_{\gat})=R(\A_{0}^{*})$ are closed and
$$N(\B_{\gat})=N(\A_{0})=\{0\},\qquad
R(\B_{\gat})=R(\A_{0}^{*})=N(\A_{0})^{\bot}=\Lt(\om),$$
which shows $\cB_{\gat}=\B_{\gat}$.
$\B_{\gat}$ is selfadjoint, positive, and bijective with bounded inverse
$$\cB_{\gat}^{-1}=\cA_{0}^{-1}(\cA_{0}^{*})^{-1}:\Lt(\om)\to D(\B_{\gat}),\qquad
\norm{\cB_{\gat}^{-1}}_{\Lt(\om)\to\Lt(\om)}\leq\cfp^{4}.$$
\eqref{eq:estAA} reads
$$\forall\,\varphi\in D(\B_{\gat})\qquad
\norm{\varphi}_{\Lt(\om)}
\leq\cfp^{2}\norm{\nana\varphi}_{\Lt(\om)}
\leq\cfp^{4}\norm{\divDiv\nana\varphi}_{\Lt(\om)}
=\cfp^{4}\norm{\Delta^{2}\varphi}_{\Lt(\om)}.$$

Let $f\in\Lt(\om)$.
Then $u:=\cB_{\gat}^{-1}f\in D(\B_{\gat})$ is the unique solution of
the Dirichlet/Neumann biharmonic boundary value problem $\B_{\gat}u=f$, i.e.,
\begin{align*}
\divDiv\nana u=\Delta^{2}u&=f
&
&\text{in }\om,\\
u&=0
&
&\text{on }\gat,\\
\na u&=0
&
&\text{on }\gat,\\
(\nana u)\,n&=0
&
&\text{on }\gan,\\
(\Div\nana u)\cdot n&=0
&
&\text{on }\gan.
\end{align*}
To find $u\in\H^{2}_{\gat}(\om)$ by variational methods we may consider \eqref{eq:varform2}, i.e.,
$$\forall\,\varphi\in\H^{2}_{\gat}(\om)\qquad
\scp{\nana u}{\nana\varphi}_{\Lt(\om)}
=\scp{f}{\varphi}_{\Lt(\om)}.$$

\section*{Acknowledgement}

For the second author it is a pleasure to thank Filippo Gazzola and Pier Domenico Lamberti 
for some useful conversations about boundary value problems for the biharmonic operator
and INdAM-GNCS for the support.

\bibliographystyle{plain} 
\bibliography{pv}

\vspace*{5mm}
\hrule
\vspace*{3mm}

\end{document}